\documentclass[a4paper,leqno,12pt]{amsart}

\usepackage{amsmath}
\usepackage{amsfonts}	
\usepackage{amssymb}  	
\usepackage{mathrsfs}
\usepackage{mathtools}	
\usepackage{savesym}
\usepackage{esvect}		
\usepackage{ifthen}		
\usepackage{color} 	 	
\usepackage{graphicx}  	
\usepackage{caption}  	
\usepackage{amscd} 	  	
\usepackage{float} 	  	
\usepackage[svgnames]{xcolor}

\usepackage[shortlabels,inline]{enumitem} 
\usepackage{tikz}    
\usepackage{tikz-cd} 

\usepackage[a4paper,%
			margin=1in,%
			top=1.2in,
			bottom=1.2in,
]{geometry}

\usepackage[numbers,sort]{natbib} 	

\usepackage{aliascnt} 
\usepackage[pdfusetitle,
		bookmarks=true,
		pagebackref=false,
		colorlinks,
		linkcolor=Navy,
		citecolor=Navy,
		urlcolor=black]{hyperref}

\usepackage{amsthm}

\theoremstyle{plain}
\newaliascnt{theorem}{dummy}
\newtheorem{theorem}[theorem]{Theorem}
\aliascntresetthe{theorem}

\newaliascnt{proposition}{dummy}
\newtheorem{proposition}[proposition]{Proposition}
\aliascntresetthe{proposition}

\newaliascnt{corollary}{dummy}
\newtheorem{corollary}[corollary]{Corollary}
\aliascntresetthe{corollary}

\newaliascnt{lemma}{dummy}
\newtheorem{lemma}[lemma]{Lemma}
\aliascntresetthe{lemma}

\newaliascnt{conjecture}{dummy}

\aliascntresetthe{conjecture}

\theoremstyle{definition}
\newaliascnt{definition}{dummy}
\newtheorem{definition}[definition]{Definition}
\aliascntresetthe{definition}

\newaliascnt{example}{dummy}
\newtheorem{example}[example]{Example}
\aliascntresetthe{example}

\theoremstyle{remark}
\newaliascnt{remark}{dummy}
\newtheorem{remark}[remark]{Remark}
\aliascntresetthe{remark}

\numberwithin{equation}{section} 

\allowdisplaybreaks

\newcommand{\calB}{\mathcal{B}}

\newcommand{\calD}{\mathcal{D}}
\newcommand{\calQ}{\mathcal{Q}}

\newcommand{\calM}{\mathcal{M}}
\newcommand{\calL}{\mathcal{L}}

\newcommand{\calW}{\mathcal{W}}
\newcommand{\calA}{\mathcal{A}}
\newcommand{\calH}{\mathcal{H}}
\newcommand{\calJ}{\mathcal{J}}

\newcommand{\scrC}{\mathscr{C}}

\newcommand{\bbR}{\mathbb{R}}

\newcommand{\bbZ}{\mathbb{Z}}
\newcommand{\bbN}{\mathbb{N}}
\newcommand{\bbP}{\mathbb{P}}
\newcommand{\bbT}{\mathbb{T}}

\DeclareMathOperator*{\intxn}{\cap}
\DeclareMathOperator*{\Intxn}{\bigcap}

\DeclareMathOperator*{\Union}{\bigcup}

\DeclareMathOperator{\Var}{Var}

\DeclareMathOperator{\diag}{diag}

\providecommand{\abs}[1]{\lvert#1\rvert}
\providecommand{\Abs}[1]{\left\lvert#1\right\rvert}

\newcommand{\euclid}[1][d]{\mathbb{R}^{#1}}
\newcommand{\torus}[1][2]{\mathbb{T}^{#1}}

\newcommand{\wt}[1]{\widetilde{#1}}

\newcommand{\dimH}{\dim_{\mathrm{H}}}

\newcommand{\dimB}{\dim_{\mathrm{B}}}

\newcommand{\uDim}[1]{\overline{\dim}_{\mathrm{#1}}}

\newcommand{\haus}[2]{\mathcal{H}^{#1}(#2)}

\newcommand{\hausO}[3]{\mathcal{H}^{#1}_{#2}({#3})}
\newcommand{\Haus}[1]{\mathcal{H}^{#1}}

\newcommand{\tk}[1][]{\lfloor \theta_{#1} k \rfloor}

\newcommand{\ff}[1]{\prod_{i=1}^{s} f_{i}(#1)^{\theta_{i-1}-1}}
\newcommand{\ffi}[1]{\#\pi_{i}^{-1}(\tau_{i}(#1))}
\newcommand{\lang}[2]{\mathcal{L}_{#1}(X_{#2})}
\newcommand{\langX}[1][]{\lang{#1}{}}

\newcommand{\mFor}{\quad\text{ for }}
\newcommand{\pmat}[1]{\begin{pmatrix}#1\end{pmatrix}}
\newcommand{\dLY}[1]{\mathrm{d}_{\mathrm{LY}}(#1)}
\newcommand{\mAnd}{\quad\text{ and }\quad}

\newcommand{\BMcarpet}[1][]{Bedford\nobreakdash-McMullen carpet#1}
\newcommand{\BMsponge}[1][]{Bedford\nobreakdash-McMullen sponge#1}

\title{On the coincidence of the Hausdorff and box dimensions for some affine-invariant sets}

\author{Zhou Feng}
\address{Department of Mathematics\\
	The Chinese University of Hong Kong\\
	Shatin,  Hong Kong}
\curraddr{}
\email{\href{mailto: zfeng@math.cuhk.edu.hk}{zfeng@math.cuhk.edu.hk}}
\thanks{}


\date{}
\dedicatory{}
\thanks{}

\begin{document}
\begin{abstract}
	Let $ K $ be a compact subset of the $d$-torus invariant under an expanding diagonal endomorphism with $ s $ distinct eigenvalues. Suppose the symbolic coding of $K$ satisfies weak specification. When $ s \leq 2 $, we prove that the following three statements are equivalent: (A) the Hausdorff and box dimensions of $ K $ coincide; (B) with respect to some gauge function, the Hausdorff measure of $ K $ is positive and finite; (C) the Hausdorff dimension of the measure of maximal entropy on $ K $ attains the Hausdorff dimension of $ K $. When $ s \geq 3 $, we find some examples in which (A) does not hold but (C) holds, which is a new phenomenon not appearing in the planar cases. Through a different probabilistic approach, we establish the equivalence of (A) and (B) for Bedford-McMullen sponges.
\end{abstract}

\maketitle

\section{Introduction}

The Hausdorff and box dimensions are the most common notions to quantify the size of sets in fractal geometry; see~\cite{Falconer2003,Mattila1995} for an introduction. It is natural to ask under which condition these two dimensions coincide, especially for invariant sets in dynamical systems. This is an important question and has a long history. Furstenberg~\cite{Furstenberg1967} showed that if $ K $ is a compact set invariant under an expanding conformal endomorphism on the $d$-torus, then
\begin{equation}\label{eq:DimCoin}\tag{A}
	\dimH K = \dimB K,
\end{equation}
where $ \dimH $ and $ \dimB $ stand for the Hausdorff and box dimensions respectively. Previously, \eqref{eq:DimCoin} has been proved to hold in various conformal settings~\cite{Falconer1989,GatzourasPeres1997,Barreira1996} and some typical nonconformal settings~\cite{Falconer1988,KaeenmaekiVilppolainen2010}. Notably, Falconer proved that \eqref{eq:DimCoin} holds for all self\nobreakdash-similar sets~\cite{Falconer1989} and some typical self-affine sets~\cite{Falconer1988}. In contrast, the Hausdorff and box dimensions are usually distinct in specific nonconformal settings. In this paper, we study when \eqref{eq:DimCoin} holds for the invariant sets in the following family of nonconformal dynamical systems. Let $ T $ be an expanding linear endomorphism on the $ d $-torus $ \torus[d] = \bbR^{d}/\bbZ^{d} $ represented by a diagonal integer matrix
\begin{equation*}
	\Lambda = \diag (m_{1}, \ldots, m_{d}),
\end{equation*}
where $ m_{1} \geq \cdots \geq m_{d} \geq 2 $, that is, $ T(x) = \Lambda x \mod 1 $ for $ x \in [0,1)^{d}$. Write the number of distinct expanding ratios as
\begin{equation}\label{eq:def-s}
	 s = \#  \{m_{i}: 1 \leq i \leq d\} ,
\end{equation}
 where $\# $ stands for cardinality. Let
$ \calA = \prod_{j=1}^{d} \{ 0, \ldots, m_{j} - 1 \} $.
There is a canonical representation map $ R \colon \calA^{\bbN} \to \torus[d] $ given by
\begin{equation}\label{def:R-repr}
	R(x) = \sum_{k=1}^{\infty} \Lambda^{-k} x_{k} \quad \text{ for } x = (x_{k})_{k=1}^{\infty} \in \calA^{\bbN}.
\end{equation}
Let $ \sigma $ be the (left) shift on $\calA^{\bbN}$ defined as  $ \sigma \left ((x_{k})_{k=1}^{\infty}\right ) = (x_{k+1})_{k=1}^{\infty} $. Endow $ \calA^{\bbN} $ with the product topology. A closed subset $ X $  of $ \calA^{\bbN}$ is called a \textit{subshift} if $\sigma(X) \subset X $. Since $ R $ is a continuous surjective map and $ R\circ \sigma = T \circ R $, there is an one-to-one correspondence between the subshifts of $ \calA^{\bbN}$ and the compact $T$-invariant subsets of $ \bbT^{d}$. When $ K  =  R(\calD^{\bbN}) $ for some nonempty subset $ \calD $ of $ \calA$, we call $K$ a \textit{Bedford-McMullen carpet} if $ d = 2$ and a \textit{Bedford-McMullen sponge} if $ d \geq 3 $.

Given a continuous increasing function $ \varphi \colon [0, \infty) \to [0, \infty) $ with $ \varphi(0) = 0 $, the Hausdorff measure $ \haus{\varphi}{E} $ for $ E \subset \euclid[d] $ with respect to the \textit{gauge function} $\varphi $ is
\begin{equation*}
	\haus{\varphi}{E} = \lim_{\delta\to 0} \, \inf \left \{ \sum_{i=1}^{\infty} \varphi(\abs{E_{i}}) \colon \{E_{i}\}_{i=1}^{\infty} \text{ is a cover of } E \text{ and } \abs{E_{i}} \leq \delta \text{ for all } i \right \},
\end{equation*}
where $ \abs{F} $ denotes the diameter of a set $ F \subset \euclid[d] $ in Euclidean metric. Write  $ \calH^{\varphi} $ as $ \calH^{\gamma} $ if $ \varphi(t) = t^{\gamma}$, $\gamma \geq 0 $. We say that there exists a gauge function for $ K \subset \euclid$ if
\begin{equation}\label{eq:ExistGauge}\tag{B}
	0 < \haus{\varphi}{K} < \infty \mFor \text{some gauge function } \varphi.
\end{equation}
In this paper, for a compact $T$-invariant set $ K $, we aim to characterize the coincidence of dimensions as in \eqref{eq:DimCoin} by the existence of gauge functions as in \eqref{eq:ExistGauge}.




Our research target is motivated as follows. In 1980s, Bedford~\cite{Bedford1984} and McMullen~\cite{McMullen1984} independently computed the Hausdorff and box dimensions of \BMcarpet[s]. Later their work was extended by Kenyon and Peres~\cite{KenyonPeres1996,KenyonPeres1996a} to \BMsponge[s] and certain sofic affine-invariant sets; see \cite{LalleyGatzouras1992,Baranski2007,FengWang2005,Fraser2012a} for some other generalizations. For a \BMcarpet\ $K$, McMullen~\cite{McMullen1984} pointed out that  $ 0 < \haus{\dimH K}{K} < \infty $ if $\dimH K = \dimB K $, and asked about the value of $\haus{\dimH K}{K}$ when $ \dimH K < \dimB K $. This question was answered by Peres \cite{Peres1994} who proved that $ \haus{\dimH K}{K} = \infty $ if $ \dimH K < \dimB K $. Afterwards Peres~\cite{Peres1994a} showed that if $ \dimH K < \dimB K $, the packing measure of $K$ at its packing dimension is also infinite. Before the result of Peres~\cite{Peres1994}, an elegant argument found independently by Bedford and Mandelbrot (unpublished) and by Gatzouras and Lalley~\cite{LalleyGatzouras1992} gives the equivalence of \eqref{eq:DimCoin} and \eqref{eq:ExistGauge} for \BMcarpet[s] (see e.g.\ \cite[Proposition 2]{Peres1994}). Actually, Gatzouras and Lalley~\cite{LalleyGatzouras1992} showed that $ \dimH K = \dimB K $ is equivalent to $ 0 < \haus{\dimH K}{K} < \infty $ for a \textit{Gatzouras\nobreakdash-Lalley carpet} $ K $ which is a generalization of \BMcarpet\ by relaxing the grid structure but keeping the dominated directions. These previous results motivate us to study the equivalence of \eqref{eq:DimCoin} and \eqref{eq:ExistGauge} for invariant subsets of \BMsponge[s] (also called \textit{sub-self-affine sets}~\cite{Falconer1995,KaeenmaekiVilppolainen2010}).

To state our results, we introduce another related condition that the Hausdorff dimension of the measure of maximal entropy attains the Hausdorff dimension of its support. For a subshift $ X $ of $ \calA^{\bbN} $, write
\begin{equation*}
	\lang{k}{} = \{ i_{1} \cdots i_{k} \in \calA^{k} \colon i_{1}\cdots i_{k} = x_{1}\cdots x_{k} \text{ for some } (x_{j})_{j=1}^{\infty} \in X \}, \quad k \in \bbN
\end{equation*}
and
\begin{equation*}
	\langX  = \Union_{k=1}^{\infty} \lang{k}{} \Union \{\varnothing\},
\end{equation*}
where $ \varnothing $ denotes the empty word. For $ I \in \calA^{k}$, denote $ \abs{I} = k $. Set $ \abs{\varnothing} = 0 $. Let $ IJ $ denote the juxtaposition of $ I, J \in \calL(X)$. We say a subshift $ X $ satisfies \textit{weak specification} if there exists an integer $ p $ such that for every $ I, J \in \calL(X) $, there is $ W \in \calL(X) $ with $ \abs{W} \leq p $ such that $ IWJ \in \calL(X) $. Let $ X $ be a subshift satisfying weak specification and $ K = R(X)$. There is a unique $\sigma$\nobreakdash-invariant \textit{measure of maximal entropy} $ \mu $ on $ X $ (see \autoref{coro:parry}). Then we give another condition that
\begin{equation}\label{eq:MeasCoin} \tag{C}
	\dimH R\mu = \dimH K,
\end{equation}
where $ R\mu = \mu \circ R^{-1} $ is the push-forward of $ \mu $ under $ R $, and $ \dimH R\mu$ is the Hausdorff dimension of $ R\mu $ defined to be the infimum of the Hausdorff dimension of Borel sets with positive $R\mu$-measure. 


Now we are ready to state our first result.

\begin{theorem}\label{thm:DimCoin}
	Let $ X $ be a subshift satisfying weak specification and $ K = R(X) $. Recall the definition of $ s $ from \eqref{eq:def-s}. Then the following statements hold.
	\begin{enumerate}[(i)]
		\item\label{itm:A=>B} \eqref{eq:DimCoin} implies $ 0 < \haus{\dimH K}{K} < \infty $, and so \eqref{eq:ExistGauge}.
		
		\item\label{itm:B=>C} \eqref{eq:ExistGauge} implies \eqref{eq:MeasCoin}.
		
		\item\label{itm:s<=2} If $ s \leq 2 $, then \eqref{eq:DimCoin}, \eqref{eq:ExistGauge}, \eqref{eq:MeasCoin} are equivalent.
		
		\item\label{itm:s>=3} If $s \geq 3 $, there are \BMsponge[s] in which \eqref{eq:DimCoin} does not hold but \eqref{eq:MeasCoin} holds.
	\end{enumerate}
\end{theorem}

In \cite[P.~526]{Peres1994}, Peres conjectured that when $d = 2$ and $ K = R(X) $ where $ X $ is a primitive subshift of finite type, $ \Haus{\dimH K}(K) = \infty $ if $ \dimH K < \dimB K $. Since such $X$ satisfies weak specification, it follows from \autoref{thm:DimCoin}\ref{itm:s<=2} that $ \Haus{\dimH K }(K) $ is either $ 0 $ or $ \infty $ if $ \dimH K < \dimB K$.

Let us give some ideas about the proof of \autoref{thm:DimCoin}. Firstly, we characterize \eqref{eq:DimCoin} in terms of the size of the fibers of certain factor maps; see \autoref{thm:Ext-KP-DimCoin}\ref{itm:UniFiber-Entropy}, which is by adapting and extending some arguments of \cite{KenyonPeres1996} to the setting of weak specification. Based on this characterization and the Gibbs property of the measure of maximal entropy, we obtain an upper bound on the density of $R\mu$, leading to the part of \ref{itm:A=>B} that $\calH^{\dimH K} > 0 $. Unlike Bedford-McMullen carpets corresponding to the full shifts, in the case of subshifts it is difficult to estimate the density of $ R\mu $ from below. Instead of following the approach based on the density to prove the other part of \ref{itm:A=>B} that $ \calH^{\dimH K} < \infty $, we directly consider the covers consisting of the approximate cubes whose number is estimated using weak specification. The proof of \ref{itm:B=>C} is based on the translation-invariance of Hausdorff measures and the Gibbs property of the measure of maximal entropy; see \autoref{prop:HausEqParry}. Next we move to the proof of \ref{itm:s<=2} which is more difficult and requires new ideas since the Hausdorff dimension of $T$-invariant sets is expressed as a certain topological pressure instead of the simple formula for Bedford-McMullen carpets. By \ref{itm:A=>B} and \ref{itm:B=>C}, it suffices to show that \eqref{eq:MeasCoin} implies \eqref{eq:DimCoin} when $ s \leq 2 $; see \autoref{lem:ParryFullDim=>DimEq}. To this end, we manage to compute the Hausdorff dimension of the invariant set based on the key observation that the projection of the measure of full dimension to the weakest unstable direction always has the Gibbs property. As for \ref{itm:s>=3}, we give some examples of \BMsponge[s] (\autoref{ex:d=3} and \autoref{ex:d>=4}) with the desired properties, which is a new phenomenon not appearing in the planar cases.

To determine the Hausdorff measures of \autoref{ex:d=3} and \autoref{ex:d>=4}, we manage to give a sufficient condition for \BMsponge[s] to have infinite Hausdorff measures (see \autoref{prop:CriterionInfHaus}), and conclude that \autoref{ex:d=3} and \autoref{ex:d>=4} have infinite Hausdorff measures at their respective Hausdorff dimensions. The proof of \autoref{prop:CriterionInfHaus} is based on the method of Peres~\cite{Peres1994} (see \autoref{prop:PeresInRd}) and a newfound relation satisfied by the measure of full dimension (see \autoref{lem:Sum-delta-i}).  It is worth pointing out that in \cite[Section 5]{Peres1994} Peres mentioned the result (\cite[Theorem 1]{Peres1994}), that $\haus{\dimH K}{K} = \infty$ if $\dimH K < \dimB K $ for a \BMcarpet\ $ K $, can extend to \BMsponge[s]. However, he did not give a detailed justification.

The equivalence of \eqref{eq:DimCoin} and \eqref{eq:ExistGauge} may still hold when $ s \geq 3 $. However, \ref{itm:s>=3} shows that our current strategy for the proof of \ref{itm:s<=2} can not extend to $ s \geq 3 $. Nevertheless, through a different probabilistic approach we obtain the equivalence of \eqref{eq:DimCoin} and \eqref{eq:ExistGauge} for \BMsponge[s].
\begin{theorem}\label{thm:SierSponge}
	 \eqref{eq:DimCoin} and \eqref{eq:ExistGauge} are equivalent for Bedford-McMullen sponges.
\end{theorem}
The proof of \autoref{thm:SierSponge} is based on the independence of the coordinates of a random sequence whose law is a Bernoulli measure on $\calA^{\bbN}$ and the density theorem for the Hausdorff measures. To the best of our knowledge, the previous known proofs for the equivalence of \eqref{eq:DimCoin} and \eqref{eq:ExistGauge} for \BMcarpet[s] essentially rely on the equivalence of \eqref{eq:DimCoin} and \eqref{eq:MeasCoin} (see e.g.\ \cite[Proposition 2]{Peres1994}), which does not work for \BMsponge[s] as seen in \autoref{ex:d=3}. We remark that a similar proof of \autoref{thm:SierSponge} based on the density theorem for the packing measures shows that for a Bedford-McMullen sponge $ K$, $\dimH K = \dimB K $ is equivalent to that, with respect to some (doubling) gauge function $ \varphi$, the packing measure of $ K $ is positive and finite. One may expect that \autoref{thm:SierSponge} can extend to some subshifts if the limit theorem used in the proof of \autoref{thm:SierSponge} holds under some weaker independence; see \autoref{rmk:Gen-LimThm} for an illustration.

The paper is organized as follows. In \autoref{sec:prelim-DimCoin}, we make some preparations for the proofs of main theorems. In \autoref{sec:Ext-KP}, we extend the results in \cite{KenyonPeres1996} to the setting of weak specification. \autoref{sec:Pf-DC=>EG} is for the proof of \autoref{thm:DimCoin}\ref{itm:A=>B}. The proofs of \ref{itm:B=>C} and \ref{itm:s<=2} of \autoref{thm:DimCoin} are given in \autoref{sec:EG=>DC}. \autoref{sec:SierSponge} is devoted to the results about \BMsponge[s] including the proof of \autoref{thm:SierSponge}. The examples for \autoref{thm:DimCoin}\ref{itm:s>=3} are given in \autoref{subsec:examples}.

\section{Preliminaries}\label{sec:prelim-DimCoin}

In \autoref{subsec:subshifts}, we introduce the subshifts and factor maps. Then we present some results about the measures of maximal entropy and full dimension respectively in \autoref{subsec:Gibbs} and \autoref{subsec:FullDim}. The last subsection is devoted to the approximate cubes and the related results.

Throughout this paper we shall mean by $ a \lesssim b $ that $ a \leq Cb $ for some positive constant $ C $, and write $ a \approx b $ if $ a \lesssim b $ and $ b \lesssim a $. 



\subsection{Subshift and factor map}\label{subsec:subshifts}

Let $\calB$ be a finite set.  Let $ k \in \bbN $ and $ 0 \leq a < b \leq k $. For $ I = i_{1}\cdots i_{k} \in \calB^{k} $, define $ I|_{a}^{b} = i_{a+1}\cdots i_{b}$, $ \abs{I} = k $, and $ [I] = \{ x \in \calB^{\bbN} \colon x|k = I \}$. For $ I \in \calA^{k_{1}}, J \in \calA^{k_{2}}$, let $ IJ $ denote the juxtaposition of $ I$ and $J$. For $ x = (x_{i})_{i=1}^{\infty} \in \calB^{\bbN}$, write $ x|_{a}^{b} = x_{a+1} \ldots x_{b}$ and  $ x|k = x_{1}\cdots x_{k} $. By convention we set $ x|0 = \varnothing $ and $\abs{\varnothing} = 0 $, where $\varnothing $ denotes the empty word. Endow $\calB^{\bbN}$ with the product topology. Let $ \sigma $ be the shift on $\calB^{\bbN}$ defined by $ \sigma ( (x_{k})_{k=1}^{\infty} ) = (x_{k+1})_{k=1}^{\infty} $. A closed subset $ X $ of $\calB^{\bbN}$ is called a \textit{subshift} of $ \calB^{\bbN}$ if $\sigma(X) \subset X $. Let  $ \calM_{\sigma}(X) $ denote the set of $\sigma$-invariant measures on $ X $. Define
\begin{equation*}
	\lang{k}{} = \{ I \in \calB^{k} \colon [I] \intxn X \neq \emptyset \} \text{, }\,\, k\in\bbN\quad \text{ and }\quad \langX  = \Union_{k=1}^{\infty} \lang{k}{} \Union \{\varnothing\}.
\end{equation*}

\begin{definition}[weak specification] \label{def:WeakSpec}
	A subshift $ X $ of $ \calB^{\bbN}$ is said to satisfy \textit{weak specification} if there exists an integer $ p $ such that for every $ I, J \in \calL(X) $, there is $ W \in \calL(X) $ with $ \abs{W} \leq p $ such that $ IWJ \in \calL(X) $.
\end{definition}

Throughout this paper, by a subshift we mean a subshift of $ \calA^{\bbN}$ by default. Recall $\calA = \prod_{j=1}^{d}\{ 0, \ldots, m_{j}-1 \}$ and $ s = \# \{ m_{i} \colon 1 \leq i \leq d \} $. Since $ m_{1} \geq \cdots \geq m_{d} \geq 2 $, there are integers  $ n_{1} > \cdots > n_{s} $ and $ 0 = d_{0} <  d_{1} < \cdots < d_{s} = d $ such that for $ 1 \leq i \leq s $
\begin{equation}\label{eq:def-ni}
	m_{j} = n_{i}, \quad d_{i-1} < j \leq d_{i}.
\end{equation}
By convention we set $n_{0} = \infty$ and $\calA_{0} = \calA $.
For $ 1 \leq i \leq s $, define
\begin{equation*}
	\calA_{i} = \prod_{j>d_{i-1}} \{0, \ldots, m_{j}-1\}.
\end{equation*}
Let $ \pi_{1} \colon \calA \to \calA_{1}$ be the identity map. For $ 2 \leq i \leq s $, let $ \pi_{i} \colon \calA_{i-1} \to \calA_{i} $ be the projection by removing first $ (d_{i-1} - d_{i-2}) $ coordinates, that is,
\begin{equation*}
	\pi_{i}( (x_{d_{i-2}+1}, \ldots, x_{d_{i-1}}, x_{d_{i-1}+1}, \ldots, x_{d})) = (x_{d_{i-1}+1}, \ldots, x_{d}).
\end{equation*}
Naturally extend $ \pi_{i} $ from $ \calA_{i-1}$ to $\calA_{i-1}^{\bbN}$ by 
\begin{equation*}
	\pi_{i}(x) = (\pi_{i}(x_{k})) \mFor x = (x_{k}) \in \calA_{i-1}^{\bbN}.
\end{equation*}
For $ 1 \leq i \leq s $, define $ \tau_{i} \colon \calA^{\bbN} \to \calA_{i}^{\bbN}$ by $ \tau_{i} = \pi_{i} \circ \cdots \circ \pi_{1} $. By abuse of notation, for $1 \leq i \leq s $ we keep using $\sigma$ to denote the shift on $\calA_{i}^{\bbN}$, which shall not cause confusion since the domain of $ \sigma$ is always clear from the context. Then $ \pi_{i} \circ \sigma = \sigma \circ \pi_{i} $ and $ \tau_{i} \circ \sigma = \sigma\circ \tau_{i} $, that is, $ \pi_{i}$ and $\tau_{i}$ are factor maps. For a subshift $ X $ and $ 1 \leq i \leq s$, define a subshift of $ \calA_{i}^{\bbN}$ by $ X_{i} = \tau_{i}(X) $. Hence the following diagram commutes.
\begin{equation*}
	\begin{tikzcd}
		X \arrow{r}{\pi_{1}} \arrow[bend left=30]{r}{\tau_{1}} \arrow[bend left=47]{rr}{\tau_{i}} \arrow[bend left=52]{rrr}{\tau_{s}} & X_{1} \arrow{r}{\pi_{2}}  & \cdots X_{i} \cdots \arrow{r}{\pi_{s}} &  X_{s}
	\end{tikzcd}
\end{equation*}
For $1 \leq i \leq s $, by \autoref{def:WeakSpec} and the surjectivity of $ \tau_{i}$, if $ X $ satisfies weak specification, then so does $ X_{i} $.

\subsection{Measure of maximal entropy} \label{subsec:Gibbs}
Let $ \calB $ be a finite set and $ X $ be a subshift of $ \calB^{\bbN}$. The \textit{topological entropy} $ h(X)$ of $ X $ is 
\begin{equation}\label{eq:Def-TopEntropy}
	h(X) = \lim_{k\to\infty} \frac{1}{k} \log \# \calL_{k}(X),
\end{equation}
where the limit exists by a subadditivity argument. For $ \mu \in \calM_{\sigma}(X) $, the \textit{measure\nobreakdash-theoretic entropy} $h(\mu)$ of $\mu$ is
\begin{equation}\label{eq:def-MeasureEntropy}
	h(\mu) = \lim_{k\to\infty} \frac{1}{k}  \sum_{I \in \calL_{k}(X)} - \mu(I) \log \mu(I),
\end{equation}
where the limit exists by the invariance of $ \mu $ and a subadditivity argument. Note that $ h(\mu) \leq h(X)$ since $ \sum_{I \in \calL_{k}(X)} - \mu(I) \log \mu(I) \leq \log \# \calL_{k}(X)$ by convexity. We call $\mu$ a \textit{measure of maximal entropy} if $h(\mu) = h(X)$.

Suppose $X$ satisfies weak specification. Applying \cite[Theorem 5.5]{Feng2011} with $ \phi \equiv 1 $ on $ X $ shows that there is a unique measure of maximal entropy with the Gibbs property, an extension of the classical result of Parry \cite{Parry1964}.

\begin{proposition} \label{coro:parry}
	Let $ X $ be a subshift satisfying weak specification. Then there is a unique measure of maximal entropy $ \mu $ on $ X $. Moreover, $ \mu $ is the unique ergodic measure satisfying the following Gibbs property
	\begin{equation*}
		\mu(I) \approx \frac{1}{\#\calL_{\abs{I}}(X)} \approx \exp \left (- \abs{I}\, h(X)\right )  \quad \text{ for } \, I \in \calL(X),
	\end{equation*}
where $ h(X) $ is as in \eqref{eq:Def-TopEntropy}.
\end{proposition}

\subsection{Measure of full dimension} \label{subsec:FullDim} Roughly speaking, a measure of full dimension is an invariant measure having the same Hausdorff dimension as its support. It is an important topic in dynamical systems to determine the existence and uniqueness of the measure of full dimension on invariant sets~\cite{GatzourasPeres1997,Olivier2010,KenyonPeres1996a,Rams2005/06,Luzia2010,Feng2011,BarralFeng2011}; see \cite{GatzourasPeres1996} for a survey. In their celebrated work, Das and Simmons~\cite{DasSimmons2017} contructed a self-affine sponge whose Hausdorff dimension is strictly greater than the supremum of the Hausdorff dimensions of its invariant measures, thus showing that the measure of full dimension may not exist in an expanding nonconformal system. For other dimensional properties of invariant measures in nonconformal settings, very recently Kolossv\'ary~\cite{Kolossvary2023} computed the $L^{q}$ dimensions of self-affine measures  on self-affine sponges assuming a suitable separation condition, which generalizes the results in \cite{King1995,Olsen1998}.

For a subshift $X $ and $K = R(X)$, we call an invariant measure $ \mu $ on $ X $ a \textit{measure of full dimension} if $ \dimH R\mu = \dimH K $. Kenyon and Peres~\cite{KenyonPeres1996a} established a Ledrappier\nobreakdash-Young formula for the Hausdorff dimensions of ergodic measures and proved the existence of the measure of full dimension. For $ \mu \in \calM_{\sigma}(\calA^{\bbN}) $, following \cite{KenyonPeres1996a} we define
\begin{equation}\label{eq:DefWeightedEntropy}
	 \dLY{\mu} = \sum_{i=1}^{s} \left ( \frac{1}{\log n_{i}} - \frac{1}{\log n_{i-1}}\right ) h(\tau_{i}\mu).
\end{equation}

\begin{proposition}[{\cite[Theorem 1.1 \& Lemma 4.3]{KenyonPeres1996a}}]\label{thm:dimLY-ExistFullDim}
	Let $ X $ be any subshift and $ K = R(X)$. Then for each ergodic measure $ \mu \in \calM_{\sigma}(X) $,
	\begin{equation}\label{eq:DimLY}
		\dimH R\mu = \dLY{\mu},
	\end{equation}
	and
	\begin{equation}\label{eq:ExistFullDim}
		\dimH K  = \max_{\mu\in\calM_{\sigma}(X)} \dLY{\mu}.
	\end{equation} 
\end{proposition}

By studying the optimization problem in \eqref{eq:ExistFullDim}, a part of the so-called \textit{weighted theormodynamic formalism}, Feng \cite{Feng2011} proved that when $ X $ satisfies weak specification, there is a unique measure of full dimension with some lower Gibbs property. Before stating this result in \autoref{prop:FullDimMeas}, we need some notation. For $ 1\leq i \leq s $, write
\begin{equation}\label{eq:def-alpha}
	\alpha_{i} = \frac{\log n_{i}}{\log n_{i-1}}.
\end{equation}
Let $X$ be a subshift. Define $ \phi^{(1)} \equiv 1 $ on $ \calL(X) $. For $ 2 \leq i \leq s $, recursively define $ \phi^{(i)} \colon \calL(X_{i}) \to (0,\infty)$ by
\begin{equation}\label{eq:def-recur-phi}
	\phi^{(i)}(J) = \sum_{I \in \calL(X_{i-1}) \colon \pi_{i}(I) = J } \phi^{(i-1)}(I)^{\alpha_{i-1}} \mFor J \in \calL(X_{i}).
\end{equation}
Then define $  Z \colon \bbN \to (0, \infty) $ as
\begin{equation*}
	Z(k) = \sum_{I \in \calL_{k}(X_{s})} \phi^{(s)}(I)^{\alpha_{s}},
\end{equation*}
and
\begin{equation}\label{eq:def-WTopPressure}
	P = \lim_{k\to\infty} \frac{\log Z(k)}{k},
\end{equation}
where the limit exists by a subadditivity argument.

\begin{proposition}[{\cite[Theorems 5.5 \& 7.3]{Feng2011}}] \label{prop:FullDimMeas}
	Let $ X $ be a subshift satisfying weak specification and $ K = R(X)$. Then there is a unique measure of full dimension $ \eta $ on $X$, that is,
	\begin{equation*}
		\dimH R\eta = \dLY{\eta} = \max_{\mu\in\calM_{\sigma}(X)} \dLY{\mu} = \dimH K.
	\end{equation*}
	Moreover, $ \eta $ is ergodic and has the following lower Gibbs property,
	\begin{equation}\label{eq:HalfGibbs}
		\eta(I) \gtrsim  \psi(I) \quad \text{ for } I \in \calL(X),
	\end{equation}
	where
	\begin{equation*}
		\psi(I) = \frac{1}{Z(\abs{I})} \prod_{i = 1}^{s} \phi^{(i)}(\tau_{i}(I))^{\alpha_{i}-1}.
	\end{equation*}
	Additionally, $ \tau_{s}\eta $ has the following Gibbs property,
	\begin{equation}\label{eq:GibbsProjMeas}
		\tau_{s}\eta(J) \approx \frac{\phi^{(s)}(J)^{\alpha_{s}}}{\exp(\abs{J}P)} \quad \text{ for } J \in \calL(X_{s}),
	\end{equation}
	where $P$ is as in \eqref{eq:def-WTopPressure}.
\end{proposition}

Combining \autoref{thm:dimLY-ExistFullDim} and  \cite[Proposition 3.7]{Feng2011} (a relative variational principle for subadditive potentials) gives a formula for the Hausdorff dimension of all compact $T$\nobreakdash-invariant sets, an analog of which for certain sofic affine\nobreakdash-invariant sets was previously established in \cite[Theorem 1.1]{KenyonPeres1996}.

\begin{lemma}\label{prop:HausDim}
	Let $ X $ be a subshift and $ K = R(X) $. Then
	\begin{equation}\label{eq:HausDim}
		\dimH K = \frac{P}{\log n_{s}}
	\end{equation}
	where $P$ is as in \eqref{eq:def-WTopPressure}.
\end{lemma}

\subsection{Approximate cube}
In this subsection, we introduce the approximate cubes with side lengths $ \approx n_{s}^{-k} $ ($k\in\bbN$) by `slowing down' the iterations along the stronger unstable directions.
For $ 1 \leq i \leq s $, write
\begin{equation}\label{eq:def-theta}
	\theta_{i} = \frac{\log n_{s}}{\log n_{i}}.
\end{equation}
For $ k \in \bbN $, define $R_{k} \colon \Union_{n \geq k}\calA^{n} \to \bbT^{d} $ by
\begin{equation}\label{eq:def-Rk}
	R_{k}(I) := \sum_{\ell=1}^{k} \Lambda^{-\ell} i_{\ell}  \quad \text{ for } I = i_{1} \cdots i_{k}\cdots i_{n}.
\end{equation}
Combining \eqref{def:R-repr} and \eqref{eq:def-Rk} gives
\begin{equation}\label{eq:R=Rk+kR}
	R(x) = R_{k}(x|k) + \Lambda^{-k} R(\sigma^{k}x) \mFor x \in \calA^{\bbN}.
\end{equation}
Let $ \lfloor x \rfloor $ denote the integral part of $ x \in \bbR $. For $ I \in \calA^{k} $, the $ k $-level \textit{approximate cube} is
\begin{equation}\label{eq:AppCube}
	Q_{k}(I) : = \prod_{i = 1}^{s} \prod_{j = d_{i-1} + 1}^{d_{i}} \bigg[R_{\tk[i]}(I)_{j}, R_{\tk[i]}(I)_{j} + n_{i}^{-\tk[i]} \bigg),
\end{equation}
where $ v_{j} $ denotes the $ j $-th coordinate of a vector $ v \in \euclid[d] $. For $ x \in \calA^{\bbN} $, we define $ Q_{k}(x) = Q_{k}(x|k) $. Combining \eqref{eq:def-ni}, \eqref{eq:def-theta} and \eqref{eq:AppCube} shows that  $ \abs{ Q_{k}(x) } \approx n_{s}^{-k} $.

We begin with a lemma relating the measure of approximate cubes to that of a collection of cylinders in the symbolic space. For $ x \in X $ and $ k \in \bbN $, define
\begin{equation}\label{eq:WordsInApproxCube}
	\Gamma_{k}(x) = \bigg\{ I \in \calL_{k}(X) \colon \tau_{i}(I|\tk[i]) = \tau_{i}(x|\tk[i]) \text{ for } 1 \leq i \leq s \bigg \}.
\end{equation}

\begin{lemma}\label{lem:MassAppCube}
Let $X$ be a subshift and $ \mu $ be a Borel measure on $ X $. Then for $ x \in X $,
	\begin{equation}\label{eq:MassEst}
		R\mu(Q_{k}(x))  = \sum_{I \in \Gamma_{k}(x)}  \mu(I).
	\end{equation}
\end{lemma}
\begin{proof}
	Let $ K = R(X)$. It follows from \eqref{eq:AppCube} and \eqref{eq:WordsInApproxCube} that
	\begin{equation*}
		R^{-1}(Q_{k}(x) \intxn K) \subset \Union_{I \in \Gamma_{k}(x)} [I]
	\end{equation*}
	and
	\begin{equation*}
		R^{-1}(Q_{k}(x) \intxn K ) \Intxn \, [I]  \neq \emptyset \quad \text{ for } I \in \Gamma_{k}(x).
	\end{equation*}
	Since $ R\mu $ supports on $ K $,
	\begin{equation*}
		R\mu(Q_{k}(x)) = R\mu(Q_{k}(x) \intxn K) = \sum_{I\in \Gamma_{k}(x)} \mu(I),
	\end{equation*}
	which finishes the proof.
\end{proof}

Next we give a lemma for counting the number of all $ k $-level approximate cubes. Denote
\begin{equation*}
	\calQ_{k}(X) = \{ Q_{k}(x) \colon x \in X \}.
\end{equation*}
Then $ \calQ_{k}(X) $ covers $ K $ by the $ k $-level approximate cubes. To estimate $ \#\calQ_{k}(X) $, following \cite{McMullen1984} we introduce a collection of vectors 
\begin{equation}\label{eq:Words4NumCubes}
	\begin{aligned}
		\calD_{k}(X) :=  \bigg \{ (J_{1}, \ldots, J_{s}) \in \prod_{i=1}^{s} & \calL_{\tk[i]-\tk[i-1]} (X_{i})  \colon \text{there exists } I \in  \calL_{k}(X) \bigg. \\ & \text{such that }\tau_{i}(I|_{\tk[i-1]}^{\tk[i]}) = J_{i} \text{ for } 1 \leq i \leq s \bigg\}.
	\end{aligned}
\end{equation}

\begin{lemma}\label{lem:NumCubes}
	Let  $ X $ be a subshift. Then $ \# \calQ_{k}(X) = \# \calD_{k}(X) $ for $ k \in \bbN $.
\end{lemma}

\begin{proof}
	It suffices to construct a bijective map $ f \colon \calQ_{k}(X) \to \calD_{k}(X) $. For $ Q \in \calQ_{k}(X) $, there exists $ x \in X $ such that $ Q = Q_{k}(x) $. Define
	\begin{equation*}
		f(Q) := (\tau_{1}(x|_{\tk[-1]}^{\tk[1]}), \ldots, \tau_{s}(x|_{\tk[s-1]}^{\tk[s]})).
	\end{equation*}
	Then $ f $ is injective since for $ x, y \in X $, $ Q_{k}(x) = Q_{k}(y) $ if and only if $ \tau_{i}(x|_{\tk[i-1]}^{\tk[i]}) = \tau_{i}(y|_{\tk[i-1]}^{\tk[i]}) $ for $ 1 \leq i \leq s $. Next we show that $ f $ is surjective. For $ \mathbf{J} = (J_{1}, \ldots, J_{s}) \in \calD_{k}(X) $, by \eqref{eq:Words4NumCubes} there is $ I \in \calL_{k}(X) $ such that $ \tau_{i}(I|_{\tk[i-1]}^{\tk[i]}) = J_{i} $ for $ 1 \leq i \leq s $. Take some $ x \in [I] \intxn X $, then $ f(Q(x)) = \mathbf{J} $.
\end{proof}

An upper bound on the upper box dimension of $ R(X) $ follows immediately.
\begin{lemma}\label{prop:Box-UB}
	Let  $ X $ be a subshift and $ K = R(X) $. Then
	\begin{equation*}
		\uDim{B} K \leq \sum_{i=1}^{s}\left (\frac{1}{\log n_{i}}-\frac{1}{\log n_{i-1}}\right) h(X_{i}).
	\end{equation*}
\end{lemma}

\begin{proof}
	By definition of $ \calD_{k}(X) $,
	\begin{equation}\label{eq:UB-NumVects}
		\log \# \calD_{k}(X) \leq \log \# \left (\prod_{i=1}^{s} \calL_{\tk[i]-\tk[i-1]} (X_{i})  \right ) \leq \sum_{i=1}^{s} \log \#  \calL_{\tk[i]-\tk[i-1]} (X_{i}).
	\end{equation}
	For $ 1\leq i \leq s $, since $ h(X_{i}) = \lim_{k\to\infty} (1/k) \log \# \calL_{k}(X_{i}) $
	and $ \theta_{i} - \theta_{i-1} > 0 $,
	\begin{equation}\label{eq:UB-by-TopEntropy}
		\begin{aligned}
		\lim_{k\to\infty} \frac{\log \#  \calL_{\tk[i]-\tk[i-1]} (X_{i})}{k} & = \lim_{k\to\infty} \frac{\tk[i] - \tk[i-1]}{k} \cdot \frac{\log \#  \calL_{\tk[i]-\tk[i-1]} (X_{i})}{ \tk[i] - \tk[i-1] } \\
		                                                                     & = (\theta_{i}-\theta_{i-1}) h(X_{i}).
	\end{aligned}
	\end{equation}
	Then
	\begin{align*}
		\uDim{B} K & \leq \limsup_{k\to\infty} \frac{\log \# \calQ_{k}(X) }{ \log n_{s}^{k} }  \\ & = \limsup_{k\to\infty} \frac{\log \# \calD_{k}(X) }{ \log n_{s}^{k} }  & \text{by \autoref{lem:NumCubes}}\\
		           & \leq \frac{1}{\log n_{s}} \sum_{i=1}^{s} \limsup_{k\to\infty} \frac{\log \#  \calL_{\tk[i]-\tk[i-1]} (X_{i})}{k} & \text{by } \eqref{eq:UB-NumVects}\\
		           & = \frac{1}{\log n_{s}} \left (\sum_{i=1}^{s}(\theta_{i}-\theta_{i-1}) h( X_{i} )\right ) & \text{by } \eqref{eq:UB-by-TopEntropy},
	\end{align*}
	which finishes the proof.
\end{proof}

We end this section with following variant of the Rogers\nobreakdash-Taylor theorem~\cite{RogersTaylor1961} which is a useful tool to estimate the Hausdorff measure.

\begin{lemma} \label{lem:RogersTaylor}
	Let $ E \subset \torus[d] $ be a Borel set and $ \nu $ be a Borel measure with $ 0 < \nu(E) < \infty $. Then for a gauge function $ \varphi $,
	\begin{enumerate}[(i)]
		\item if $ \limsup_{k\to\infty}  \nu(Q_{k}(x))/\varphi(n_{s}^{-k}) \leq C $ for $ \nu $-a.e.\ $ x \in E $, then $ \haus{\varphi}{E} \gtrsim \nu(E)/C $.
		\vskip 3pt
		\item if $ \limsup_{k\to\infty} \nu(Q_{k}(x))/\varphi(n_{s}^{-k}) \geq C $ for all $ x \in E $, then $ \haus{\varphi}{E} \lesssim \nu(E) / C $.
	\end{enumerate}
\end{lemma}

%
%

\section{Some characterizations for the coincidence of dimensions} \label{sec:Ext-KP}

In this section we prove the following theorem which provides some dynamical characterizations of the coincidence of the Hausdorff and box dimensions for $ K = R(X) $ when $ X $ is a subshift satisfying weak specification.

\begin{theorem}\label{thm:Ext-KP-DimCoin}
	Let $ X $ be a subshift satisfying weak specification and $ K = R(X) $. For $1 \leq i \leq s$, let $ \mu_{i}$ be the measure of maximal entropy on $ X_{i} $ (see \autoref{coro:parry}). Then the following statements are equivalent.
	\begin{enumerate}[(a)]
		\item\label{itm:DimEq} $ \dimH K = \dimB K $. %

		\item\label{itm:ParryChain} $ \tau_{i} \mu_{1} = \mu_{i} $ for $ 1 \leq i \leq s $.
		
		\item \label{itm:UniFiber-Entropy} For $ 1 \leq i \leq s $ and $ I \in \calL(X_{i}) $, $ \# \tau_{i}^{-1}(I) \approx \exp \left (\abs{I} (h(X_{1})-h(X_{i}))\right) $, where $ h(X_{i}) $ is as in \eqref{eq:Def-TopEntropy}.

		\item \label{itm:UniFiber-lambda} There exist $ 0 \leq \lambda_{1} \leq \cdots \leq \lambda_{s} $ such that for $1 \leq i \leq s $ and $ I \in \calL(X_{i})$, $ \# \tau_{i}^{-1}(I) \approx \exp (\lambda_{i}\abs{I}) $.
	\end{enumerate}
\end{theorem}

\autoref{thm:Ext-KP-DimCoin} is proved by adapting and extending some ideas in \cite{KenyonPeres1996}. When $ X $ is an irreducible subshift of finite type, Feng, Lo and Shen \cite[Theorem 1.7]{FengEtAl2020} provided an algorithm to decide whether \ref{itm:UniFiber-lambda} holds.

\subsection{The equivalence of \ref{itm:DimEq} and \ref{itm:ParryChain}}

We first compute the box dimension, which relies on the next lemma about the number of approximate cubes.

\begin{lemma}\label{lem:CardQk-Wsp} Let $ X $ be a subshift satisfying weak specification. Then for $ k \in \bbN $,
	\begin{equation*}
		\# \calQ_{k}(X) \approx \exp\left (k \sum_{i=1}^{s} (\theta_{i} - \theta_{i-1}) h(X_{i})\right ).
	\end{equation*}
\end{lemma}
\begin{proof} For $ 1\leq i \leq s $, applying \autoref{coro:parry} to $ X_{i} $ gives
	\begin{equation}\label{eq:CardLang}
		\#\lang{k}{i} \approx \exp(k h(X_{i})).
	\end{equation}
	 By \autoref{lem:NumCubes} and \eqref{eq:Words4NumCubes},
	 \begin{equation*}
	 	\# \calQ_{k}(X) = \# \calD_{k}(X) \leq \prod_{i=1}^{s} \# \calL_{\tk[i]-\tk[i-1]} (X_{i}) \approx  \exp\left (k \sum_{i=1}^{s} (\theta_{i} - \theta_{i-1}) h(X_{i})\right ).
	 \end{equation*}
	 
 	Next we control $ \# \calQ_{k}(X) $ from below. Let $ p $ be as in \autoref{def:WeakSpec}. For $ k \in \bbN $ large, define 
 	\begin{equation}\label{eq:def-calJk}
 		\calJ_{k} = \prod_{i=1}^{s} \calL_{\tk[i]-\tk[i-1]-p}(X_{i}),
 	\end{equation}
 	and
 	\begin{equation*}
 		\calW = \left \{\prod_{i=1}^{s} (W_{i}^{l}, W_{i}^{r}) \in \prod_{i=1}^{s}\calL(X) \times \calL(X) \colon \abs{W_{i}^{l}} + \abs{W_{i}^{r}} = p \right \}.
 	\end{equation*}
 	We proceed to construct an map $ g $ from a subset of $ \calJ_{k} \times  \calW $ to $ \calD_{k}(X) $.
 	For $  \mathbf{J} = (J_{1}, \ldots, J_{s}) \in \calJ_{k} $, by the surjectivity of $\tau_{i}$ there exists $ \mathbf{I} = (I_{1}, \ldots, I_{s}) \in \calL(X)^{s}$ such that $ \tau_{i}(I_{i}) = J_{i} $ for $ 1 \leq i \leq s $. Then we define $ W_{i}^{l}$ and $ W_{i}^{r} $ recursively for $ 1 \leq i \leq s $. Take $ W_{1}^{l} = \varnothing $ and $ W_{1}^{r} \in \calL_{p}(X) $ such that $ I_{1} W_{1}^{r} \in \calL(X) $.  Suppose for some $ 1 \leq i \leq s-1 $ we have found $ W_{1}^{l}, W_{1}^{r}, \ldots, W_{i}^{l}, W_{i}^{r}$ such that 
 \begin{equation*}
 	W_{1}^{l}I_{1}W_{1}^{r} \cdots W_{i}^{l}I_{i}W_{i}^{r} \in \calL(X) \mAnd \abs{W_{j}^{l}} + \abs{W_{j}^{r}} = p \mFor 1\leq j \leq i . 
 \end{equation*}
 By \autoref{def:WeakSpec}, there exists $ W_{i+1}^{l} \in \calL(X) $ with $ \abs{W_{i+1}^{l}} \leq p $ such that $ W_{i}^{r} W_{i+1}^{l} I_{i+1} \in \calL(X) $. Take $ W_{i+1}^{r} \in \lang{p - \abs{W_{i+1}^{l}}}{} $ such that $ I_{i+1}W_{i+1}^{r} \in \calL(X) $. Hence we find $ \mathbf{W} = ( W_{1}^{l}, W_{1}^{r}, \ldots, W_{s}^{l}, W_{s}^{r}) $ such that
 	\begin{equation}\label{eq:W-Property}
 		W_{1}^{l} I_{1} W_{1}^{r} \cdots W_{s}^{l} I_{s} W_{s}^{r}  \in \calL(X) \mAnd  \abs{W_{i}^{l}} + \abs{W_{i}^{r}} = p \mFor 1 \leq i \leq s.
 	\end{equation}
 	Hence for $ \mathbf{J} \in \calJ_{k} $, we find some $ \mathbf{W} \in \calW $ as above and define
 	\begin{equation*}
 		g(\mathbf{J}, \mathbf{W}) = \prod_{i = 1}^{s} \tau_{i}(W_{i}^{l} I_{i} W_{i}^{r} ).
 	\end{equation*}
 	Denote the domain and range of $g$ by $\mathrm{Dom}(g)$ and $\mathrm{Ran}(g)$ respectively. Since $ g $ is defined for all $ \mathbf{J} \in \calJ_{k} $, we have $ \# \mathrm{Dom}(g) \geq \# \calJ_{k} $. By \eqref{eq:W-Property} and \eqref{eq:Words4NumCubes}, $\mathrm{Dom}(g) \subset \calJ_{k}\times\calW $ and $\mathrm{Ran}(g) \subset \calD_{k}(X) $. Note that $  \# g^{-1}( \mathbf{ U}) \leq \# \calW $ for $ \mathbf{U} \in \mathrm{Ran}(g) $ by the definition of $ g $. Hence
 	\begin{align*}
 	\# \calQ_{k}(X) & = \# \calD_{k}(X) &  \text{by \autoref{lem:NumCubes}} \\ 
 	& \geq \# \mathrm{Ran}(g) & \text{by } \mathrm{Ran}(g) \subset \calD_{k}(X) \\ 
 	& \geq \frac{\# \mathrm{Dom}(g)}{\# \calW } & \text{by } \# g^{-1}( \mathbf{U}) \leq \# \calW,\, \mathbf{U} \in \mathrm{Ran}(g) \\
 	& \geq \frac{\# \calJ_{k}}{\# \calW } & \text{by } \# \mathrm{Dom}(g) \geq \# \calJ_{k} \\
 		& \gtrsim \prod_{i=1}^{s} \# \calL_{\tk[i]-\tk[i-1]-p}(X_{i}) & \text{by } \eqref{eq:def-calJk} \text{ and } \# \calW
 		\leq (\# \calA)^{2p}\\
 		& \approx \exp\left (k \sum_{i=1}^{s} (\theta_{i} - \theta_{i-1}) h(X_{i})\right ) & \text{by } \eqref{eq:CardLang},
 	\end{align*}
 	for $ k $ large enough depending on $p$.
\end{proof}

Since $ \calQ_{k}(X) $ covers $ K $ by approximate cubes with side length $ \approx n_{s}^{-k} $,  \autoref{lem:CardQk-Wsp} implies the following formula for the box dimension of $K$.

\begin{proposition}\label{prop:BoxDim}
	Let $ X $ be a subshift satisfying weak specification and $ K = R(X)$. Then
	\begin{equation}\label{eq:BoxDim}
		\dimB K = \sum_{i=1}^{s} \left ( \frac{1}{\log n_{i}} - \frac{1}{\log n_{i-1}}\right ) h(X_{i})
	\end{equation}
	where $ h(X_{i}) $ is as in \eqref{eq:Def-TopEntropy}.
\end{proposition}

It is worth pointing out that recently Jurga~\cite{Jurga2023} showed that the box dimension of $ K = R(X) $ may not exist if $ X $ does not satisfy weak specification. 

The equivalence of \ref{itm:DimEq} and \ref{itm:ParryChain} follows immediately.

\begin{proof}[{Proof of \ref{itm:DimEq} $ \iff $  \ref{itm:ParryChain}}]
	By \autoref{thm:dimLY-ExistFullDim}, let $ \eta $ be a ergodic measure on $ X $ such that $ \dLY{\eta} = \dimH K $. By \eqref{eq:DefWeightedEntropy} and \autoref{prop:BoxDim}, \ref{itm:DimEq} is equivalent to
	\begin{equation}\label{eq:Diff-MeasTopEntropy}
		 \sum_{i=1}^{s} \left (\frac{1}{n_{i}} - \frac{1}{n_{i-1}}\right ) \left  ( h(X_{i}) - h(\tau_{i} \eta) \right  ) = 0.
	\end{equation}
	Since $ h(\tau_{i}\eta) \leq  h(X_{i}) $ and $ n_{i} < n_{i-1} $ for $ 1\leq i \leq s $, \eqref{eq:Diff-MeasTopEntropy} holds if and only if
	\begin{equation*}
		h(X_{i}) = h(\tau_{i}\eta) \mFor 1 \leq i \leq s.
	\end{equation*}
	The proof is finished by \autoref{coro:parry}.
\end{proof}

\subsection{The equivalence of \ref{itm:ParryChain}, \ref{itm:UniFiber-Entropy} and \ref{itm:UniFiber-lambda}}

\begin{proof}[{Proof of the equivalence of \ref{itm:ParryChain}, \ref{itm:UniFiber-Entropy} and \ref{itm:UniFiber-lambda}}] By \autoref{coro:parry}, for $ 1\leq i \leq s $ and $ I \in \lang{k}{i} $,
	\begin{equation}\label{eq:CardLang2}
		\mu_{i}(I) \approx \frac{1}{\# \lang{k}{i}} \approx \exp(-k h(X_{i})).
	\end{equation}
	Fix any $ i \in \{1, \ldots , s\} $ and $ k \in \bbN $.
	
    If \ref{itm:ParryChain} holds, then for $ I \in \calL_{k}(X_{i})$,
	\begin{equation*}
	\exp(-k h(X_{i})) \approx \mu_{i}(I) = \tau_{i}\mu_{1}(I) = \sum_{J \in \tau_{i}^{-1}(I) } \mu(J) \approx \# \tau_{i}^{-1}(I) \cdot \exp(-k h(X_{1})).
	\end{equation*}
	This implies \ref{itm:UniFiber-Entropy}.
	
	It is immediate that \ref{itm:UniFiber-lambda} follows from \ref{itm:UniFiber-Entropy} by taking $ \lambda_{i} = h(X_{1}) - h(X_{i})$.
	
	
	Next we show that \ref{itm:UniFiber-lambda} implies \ref{itm:ParryChain}. By \eqref{eq:CardLang2} and \ref{itm:UniFiber-lambda}, for $ I \in \calL_{k}(X_{i})$,
	\begin{equation}\label{eq:reg-tau-i}
		\tau_{i}\mu_{1}(I) = \sum_{J \in \tau_{i}^{-1}(I)} \mu_{1}(I) \approx \# \tau_{i}^{-1}(I) \cdot \exp(-k h(X_{1})) \approx \exp(k(\lambda_{i}-h(X_{1}))).
	\end{equation}
	Then
	\begin{equation*}
		 \# \calL_{k}(X_{i}) \cdot \exp(k(\lambda_{i}-h(X_{1}))) \approx \sum_{I \in \calL_{k}(X_{i})} \tau_{i}\mu_{1}(I) = 1.
	\end{equation*}
	Since $\# \calL_{k}(X_{i}) \approx \exp(kh(X_{i})) $ by \eqref{eq:CardLang2}, we have $ \lambda_{i} =  h(X_{1}) - h(X_{i}) $. Combining \eqref{eq:reg-tau-i} and \eqref{eq:CardLang2} gives
	\begin{equation*}
		\tau_{i}\mu_{1}(I) \approx \exp(-kh(X_{i})) \approx \mu_{i}(I).
	\end{equation*}
	This shows \ref{itm:ParryChain} by \autoref{coro:parry}.
\end{proof}



\section{Proof of \autoref{thm:DimCoin}\ref{itm:A=>B}}\label{sec:Pf-DC=>EG}

This section is devoted to the proof of  \autoref{thm:DimCoin}\ref{itm:A=>B}. We begin with a lemma controlling the number of cylinders in the approximate cube by the topological entropies.
\begin{lemma}\label{lem:Gamma-UB}
	Let $ X $ be a subshift satisfying weak specification and $ K = R(X)$.
	If $ \dimH K = \dimB K $, then for $ x \in X $ and $ k \in \bbN $,
	\begin{equation}\label{eq:Gamma-UB}
		\# \Gamma_{k}(x) \lesssim \exp\bigg ( k \sum_{i=1}^{s} (\theta_{i} - \theta_{i-1}) (h(X)-h(X_{i})) \bigg),
	\end{equation}
	where $ \Gamma_{k}(\cdot) $ is as in \eqref{eq:WordsInApproxCube}.
\end{lemma}
\begin{proof}
	For $1 \leq i \leq j \leq s $ and $ I, J \in \calL(X) $, $ \tau_{i}(I) = \tau_{i}(J) $ implies $ \tau_{j}(I) = \tau_{j}(J) $ since $ \tau_{j} = \pi_{j} \circ \cdots \circ \circ \pi_{i+1} \circ \tau_{i} $. Then by \eqref{eq:WordsInApproxCube},
	\begin{align*}
		\# \Gamma_{k}(x) & = \# \bigg\{ I \in \calL_{k}(X) \colon \tau_{i}(I|\tk[i]) = \tau_{i}(x|\tk[i]) \text{ for } 1 \leq i \leq s \bigg \}                              \\
		                 & = \# \bigg\{ I \in \calL_{k}(X) \colon \tau_{i}(I|_{\tk[i-1]}^{\tk[i]} ) = \tau_{i}(x|_{\tk[i-1]}^{\tk[i]}) \text{ for } 1 \leq i \leq s \bigg \} \\
		                 & \leq \prod_{i=1}^{s} \# \tau_{i}^{-1}\left ( \tau_{i}(x|_{\tk[i-1]}^{\tk[i]}) \right )                                                                                      \\
		                 & \approx \prod_{i = 1}^{s} \exp\bigg (k (\theta_{i} - \theta_{i-1}) (h(X)-h(X_{i}))  \bigg)
	\end{align*}
	where the last equality is by \ref{itm:UniFiber-Entropy} of \autoref{thm:Ext-KP-DimCoin}.
\end{proof}

%
%

Now we are ready to prove \autoref{thm:DimCoin}\ref{itm:A=>B}.

\begin{proof}[Proof of \autoref{thm:DimCoin}\ref{itm:A=>B}] Suppose $ \dimH K = \dimB K = \gamma $. Then \autoref{prop:BoxDim} shows
	\begin{equation}\label{eq:ValGamma}
		\begin{aligned}
			\gamma & = \frac{1}{\log n_{s}} \sum_{i=1}^{s}(\theta_{i} - \theta_{i-1})h(X_{i}).
		\end{aligned}
	\end{equation}
	
	We first show $  \haus{\gamma}{K} > 0  $. By \autoref{coro:parry}, let $ \mu $ be the unique measure of maximal entropy on $ X $. Then
	\begin{equation}\label{eq:muParry}
		\mu(I) \approx \exp(-k h(X)) \quad \text{for } I \in \calL_{k}(X).
	\end{equation}
	For $ x \in X $ and $ k \in \bbN $,
	\begin{align*}
		R\mu(Q_{k}(x)) & = \sum_{I \in \Gamma_{k}(x)}  \mu(I)                                                                                                                                                  & \text{by \autoref{lem:MassAppCube}} \\
		               & \approx \exp(-k h(X)) \cdot  \#\Gamma_{k}(x)                                                                                        & \text{by } \eqref{eq:muParry}         \\
		               & \lesssim \exp\left ( - k \left (h(X) - \sum_{i=1}^{s} (\theta_{i} - \theta_{i-1}) (h(X)-h(X_{i})) \right ) \right ) & \text{by \autoref{lem:Gamma-UB}} \\
		               & = \exp \left( - k \left ( \sum_{i=1}^{s}(\theta_{i}-\theta_{i-1})h(X_{i})\right )  \right) & \text{by } \theta_{s} = 1,\, \theta_{0} = 0                       \\
		               & = n_{s}^{-k\gamma} & \text{by } \eqref{eq:ValGamma}.
	\end{align*}
	This shows $ \haus{\gamma}{K} \gtrsim 1 $ by \autoref{lem:RogersTaylor}.

	Next we prove $ \haus{s}{K} < \infty $. Since $ \calQ_{k}(X) $ covers $ K $ by the $k$-level approximate cubes with side lengths $ \approx n_{s}^{-k} $, \autoref{lem:CardQk-Wsp} implies that for some $ C > 0 $,
	\begin{equation*}
		\hausO{\gamma}{Cn_{s}^{-k}}{K} \lesssim n_{s}^{- k \gamma} \cdot \# \calQ_{k}(X)  \approx n_{s}^{- k \gamma} \exp \left ( k \sum_{i=1}^{s}(\theta_{i}-\theta_{i-1}) h(X_{i})\right) = 1,
	\end{equation*}
	where the last equality is by \eqref{eq:ValGamma}. Letting $ k \to \infty $ gives $ \haus{\gamma}{K} \lesssim 1 $.
\end{proof}

\section{Proof of \ref{itm:B=>C} and \ref{itm:s<=2} in \autoref{thm:DimCoin}}\label{sec:EG=>DC}

\autoref{thm:DimCoin}\ref{itm:B=>C} is contained in \autoref{prop:HausEqParry}. \autoref{thm:DimCoin}\ref{itm:s<=2} follows from the combination of \autoref{lem:ParryFullDim=>DimEq} and \ref{itm:A=>B} and \ref{itm:B=>C} of \autoref{thm:DimCoin}.

\subsection{Proof of \autoref{thm:DimCoin}\ref{itm:B=>C}}\label{subsec:Haus=MME}
We first show that the Hausdorff measures of $R$-images of cylinders with the same length are comparable to each other. 

\begin{lemma}\label{lem:HausCylinder=}
	Let $ X $ be a subshift satisfying weak specification and $ K = R(X)$. Let $ \varphi $ be a gauge function. Then for $ k \in \bbN $ and $ I \in \calL_{k}(X)$,
	\begin{equation*}
		\haus{\varphi}{R([I]) \intxn K } \approx \haus{\varphi}{\Lambda^{-k} K}.
	\end{equation*}
\end{lemma}

\begin{proof}
	For $ J \in\calL(X)$, write $ \calD_{J} = \{ j \in \calA \colon Jj \in \calL(X)\}$.
	By the translation-invariance of $ \Haus{\varphi}$ and \eqref{eq:R=Rk+kR},
	\begin{equation}\label{eq:HausCylinder}
		\haus{\varphi}{R([J]\intxn X)} = \sum_{j\in\calD_{J}} \haus{\varphi}{R([Jj]\intxn X)} = \sum_{j\in\calD_{J}} \haus{\varphi}{\Lambda^{-\abs{J}}R([j]\intxn X)}.
	\end{equation}
	Then since $ \calD_{I} \subset \calL_{1}(X)$, $ X = \Union_{j\in\calL_{1}(X)} [j]\intxn X$, and $ K = R(X)$,
	\begin{equation}\label{eq:UB-HausCylinder}
		\haus{\varphi}{R([I]\intxn X)} \leq \sum_{j\in\calL_{1}(X)} \haus{\varphi}{\Lambda^{-k}R([j]\intxn X)} = \haus{\varphi}{\Lambda^{-k}R(X)} = \haus{\varphi}{\Lambda^{-k}K}.
	\end{equation}
	
	Next we give a lower bound on $\haus{\varphi}{R([I]\intxn X)}$. Let $ p $ be as in \autoref{def:WeakSpec}. For $ j \in \calL_{1}(X) $, there exists $ W  \in \calL(X)$ with $ \abs{W} \leq p $ such that $ I W j \in \calL(X)$. Then
	\begin{equation}\label{eq:LB-tmp}
		\begin{aligned}
			\haus{\varphi}{R([I]\intxn X)} & \geq \haus{\varphi}{R( [IWj] \intxn X ) } \\
			& \geq \haus{\varphi}{\Lambda^{-k-\abs{W}} R([j]\intxn X) } & \text{by } \eqref{eq:HausCylinder} \\
			& \geq \haus{\varphi}{\Lambda^{-k-p} R([j]\intxn X) },
		\end{aligned}
	\end{equation}
	where the last inequality is by $ \abs{W} \leq p $ and $ \Lambda^{-1}$ is contracting. Summing \eqref{eq:LB-tmp} over $ j \in \calL_{1}(X)$ gives
	\begin{equation}\label{eq:LB-cy-med}
		\haus{\varphi}{R([I]\intxn X)} \geq \frac{1}{\#\calA}\haus{\varphi}{\Lambda^{-k-p} R(X)}.
	\end{equation}
	Notice that $ \Lambda^{-k} T^{-p}(E) = \Lambda^{-k-p}E + \Lambda^{-k} \sum_{J\in \calA^{p}} R_{p}(J)$  for $ E \subset \bbT^{d}$. By the translation-invariance of $\Haus{\varphi}$,
	\begin{equation*}
		\haus{\varphi}{\Lambda^{-k-p} E} = \frac{1}{(\# \calA)^{p}} \haus{\varphi}{\Lambda^{-k}T^{-p}(E)}.
	\end{equation*}
	Then by \eqref{eq:LB-cy-med}, taking $ E = R(X) $ in the above equation implies that
	\begin{equation} \label{eq:LB-HausCylinder}
		\haus{\varphi}{R([I]\intxn X)} \geq \frac{1}{(\#\calA)^{p+1}}\haus{\varphi}{\Lambda^{-k}T^{-p}(R(X))} \geq \frac{1}{(\#\calA)^{p+1}}\haus{\varphi}{\Lambda^{-k}K},
	\end{equation}
	where the last inequality is by $ K = R(X)$ and $ K \subset T^{-p}(K) $. Combining \eqref{eq:UB-HausCylinder} and \eqref{eq:LB-HausCylinder} finishes the proof.
\end{proof}

Now we are ready to prove \autoref{thm:DimCoin}\ref{itm:B=>C} which is contained in the following proposition.

\begin{proposition}\label{prop:HausEqParry} Let $X$ be a subshift satisfying weak specification and $ K = R(X)$. Let $ \mu $ be the measure of maximal entropy on $ X $. If $ 0 < \haus{\varphi}{K} < \infty $ for some gauge function $ \varphi $, then
	\begin{equation*}
		R\mu \approx \Haus{\varphi}|_{K}
	\end{equation*}
	where $ \Haus{\varphi}|_{K} $ denotes the restriction of $ \Haus{\varphi} $ to $ K $.
	In particular, $ \dimH R\mu = \dimH K $.
\end{proposition}


\begin{proof}
	Let $ k \in \bbN $. By \autoref{lem:HausCylinder=}, 
	\begin{equation}\label{eq:TransInv}
		\haus{\varphi}{R([I])\intxn K } \approx \haus{\varphi}{R([J]) \intxn K} \mFor I, J \in \calL_{k}(X).
	\end{equation}
	Since $ 0 < \haus{\varphi}{K} < \infty $ and $ \{ R([I])\}_{I\in\calL_{k}(X)}$ covers $ K $,
	\begin{equation}\label{eq:Sum-HausPhi}
		1 \approx \Haus{\varphi}|_{K}(K) = \sum_{I \in \calL_{k}(X)} \Haus{\varphi}|_{K}(R([I])).
	\end{equation}
	Combining \eqref{eq:TransInv} and \eqref{eq:Sum-HausPhi} gives
	\begin{equation*}
		\Haus{\varphi}|_{K}(R([I])) \approx \frac{1}{\# \calL_{k}(X)} \quad \text{for } I \in \calL_{k}(X).
	\end{equation*}
	Then \autoref{coro:parry} implies that
	\begin{equation*}
		R\mu(R([I])) = \mu(I) \approx \frac{1}{\# \calL_{k}(X)} \approx \Haus{\varphi}|_{K}(R([I])) \quad \text{for } I \in \calL_{k}(X).
	\end{equation*}
	Note that $ R\mu(R([I])) = 0 = \Haus{\varphi}|_{K}(R([I])) $ for $ I \in \calA^{k}\setminus\calL_{k}(X) $ since $\mu $ supports on $X$. Hence
	\begin{equation*}
		R\mu(R([I])) \approx \Haus{\varphi}|_{K}(R([I])) \quad \text{for all } I \in \calA^{k}.
	\end{equation*}
	This finishes the proof because the collection of sets
		$\left \{ R([I]) \colon I \in \Union_{k=1}^{\infty} \calA^{k} \right \}$
	generates the Borel $ \sigma $-algebra of $ \torus[d] $.
\end{proof}

%

\subsection{Proof of \autoref{thm:DimCoin}\ref{itm:s<=2}}
The proof of \autoref{thm:DimCoin}\ref{itm:s<=2} follows immediately by combining the following proposition with \ref{itm:A=>B} and \ref{itm:B=>C} of \autoref{thm:DimCoin}.

\begin{proposition} \label{lem:ParryFullDim=>DimEq}
	  Let $ X $ be a subshift satisfying weak specification and $ K = R(X)$. Let $\mu$ be the measure of maximal entropy on $ X$.  Suppose $ s \leq 2 $. If $\dimH R\mu = \dimH K $, then $ \dimH K = \dimB K $.
\end{proposition}

\begin{proof} 
	By the definition of $ \mu $ and the assumption on $ \mu $,
	\begin{equation*}
		h(\mu) = h(X) \quad \text{ and } \quad \dimH R\mu = \dimH K.
	\end{equation*}

	When $ s = 1 $. By \autoref{prop:FullDimMeas} and \autoref{prop:BoxDim},
	\begin{equation*}
		\dimH K = \dimH R\mu = \dLY{\mu} = \frac{h(\mu)}{\log n_{1}} = \frac{h(X)}{\log n_{1}} = \dimB K.
	\end{equation*} 
	(This is also the result of Furstenberg \cite{Furstenberg1967} since $ T $ is conformal when $ s = 1 $.)

	When $ s = 2 $. Recall $ \theta_{1} = \alpha_{2} \theta_{2} = \alpha_{2} $ from \eqref{eq:def-alpha} and \eqref{eq:def-theta}. By \autoref{prop:HausDim},
	\begin{equation}\label{eq:HausDim-InPf}
		\dimH K = \frac{P}{\log n_{2} }
	\end{equation}
	where
	\begin{equation}\label{eq:s2-pressure}
		P = \lim_{k\to\infty} \frac{1}{k} \log \sum_{J \in \calL_{k}(X_{2})} (\# \pi_{2}^{-1}(J))^{\theta_{1}}.
	\end{equation}
	Let $k \in\bbN$ and $ J \in \calL_{k}(X_{2})$. Since $ \tau_{2} = \pi_{2} $, \autoref{prop:FullDimMeas} implies that
	\begin{equation}\label{eq:s2-Gibbs}
		\pi_{2}\mu(J) \approx \frac{(\# \pi_{2}^{-1}(J))^{\theta_{1}}}{\exp(kP)}.
	\end{equation}
	Since $ \mu(I) \approx \exp(-k h(X_{1})) $ for $ I \in \calL_{k}(X_{1})$ by \autoref{coro:parry},
	\begin{equation}\label{eq:s2-MaxEntropy}
		\pi_{2}\mu(J) = \sum_{I \in \pi_{2}^{-1} (J)} \mu(I) \approx  \frac{\# \pi_{2}^{-1}(J)}{ \exp(kh(X_{1}))}.
	\end{equation}
	Combining \eqref{eq:s2-Gibbs} and \eqref{eq:s2-MaxEntropy} gives
	\begin{equation}\label{eq:Fiber-Reg-Pi2}
		(\# \pi_{2}^{-1}(J))^{\theta_{1}} \approx \exp\left (k \frac{\theta_{1}}{1 - \theta_{1} } (h(X_{1}) - P)\right ).
	\end{equation}
	By \autoref{coro:parry},
	\begin{equation} \label{eq:NumWords-X2}
		\#\calL_{k}(X_{2}) \approx \exp(kh(X_{2})).
	\end{equation}
	Applying \eqref{eq:Fiber-Reg-Pi2} and \eqref{eq:NumWords-X2} to \eqref{eq:s2-pressure}, we solve that
	\begin{equation*}
		P = \theta_{1} h(X_{1}) + (1-\theta_{1}) h(X_{2}).
	\end{equation*}
	This shows $ \dimH K = \dimB K $ by \autoref{prop:BoxDim} and \eqref{eq:HausDim-InPf}.
\end{proof}

In the above proof, the assumption of $ s \leq 2 $ allows the potential $ \phi^{(2)} $ as in \eqref{eq:def-recur-phi} to be estimated by combining the Gibbs property of $ \mu $ and $\tau_{2} \mu $. However, this argument does not extend to $ s \geq 3 $ since for $ 2 \leq i \leq s - 1 $ there is a lack of the information about the regularity of $ \tau_{i} \mu $. In fact, the examples in \autoref{subsec:examples} show that \autoref{lem:ParryFullDim=>DimEq} does not hold when $ s \geq 3 $.

\section{Bedford-McMullen sponges} \label{sec:SierSponge}

This section is devoted to the results about \BMsponge[s]. The proof of \autoref{thm:SierSponge} is given in \autoref{subsec:Pf-BM}.

\subsection{Dimensional results}

In this subsection, we include some results of Kenyon and Peres about the dimensions of Bedford-McMullen sponges. Let $\calD$ be a nonempty subset of $ \calA$. Define $ \calD_{i} = \tau_{i} (\calD) $ for $ 1\leq i \leq s $. Define $ Z^{(1)} \equiv 1 $ on $\calD_{1} $. For $ 2 \leq i \leq s $, recursively define
\begin{equation}\label{eq:def-recur-Zi}
	Z^{(i)}(x) = \sum_{y\in \calD_{i-1} \colon \pi_{i}(y) = x } Z^{(i-1)}(y)^{\alpha_{i-1}} \mFor x \in \calD_{i}.
\end{equation}
Define
\begin{equation}\label{eq:def-Z-KenyonPeres}
	Z = \sum_{x \in \calD_{s}} Z^{(s)}(x)^{\alpha_{s}}.
\end{equation}


\begin{proposition}[{\cite[Theorem 1.2]{KenyonPeres1996a}}]\label{prop:Sier-HausDim}
	Let $ K = R(\calD^{\bbN}) $ for some $ \emptyset \neq \calD \subset \calA $. Then
	\begin{equation*}
		\dimH K = \frac{\log Z}{\log n_{s}}
	\end{equation*}
	where $ Z $ is as in \eqref{eq:def-Z-KenyonPeres}. There is a unique measure of full dimension $\eta_{p} $ which is the Bernoulli measure with marginal $ p = p(x)_{x\in\calD} $, where
	\begin{equation}\label{eq:McMullen-vec}
		p(x) = \frac{1}{Z} \prod_{i = 1}^{s} Z^{(i)}(\tau_{i}(x))^{\alpha_{i}-1} \mFor x \in \calD.
	\end{equation}
\end{proposition}

There is a simple criterion for \BMsponge[s] to have equal Hausdorff and box dimensions. For $ 1 \leq i \leq s $, define
\begin{equation}\label{eq:Def-Fi}
	f_{i}(x) = \ffi{x} \mFor x \in \calD.
\end{equation}
\begin{proposition}[{\cite[Proposition 1.3]{KenyonPeres1996a}}]\label{prop:DimCoin-Sier} Let $ K = R(\calD^{\bbN}) $ for some $\emptyset \neq \calD \subset \calA $. Then $\dimH K = \dimB K$ if and only if
	\begin{equation}\label{eq:UniformFiber}
		f_{i}(x) = f_{i}(y) \quad \text{ for all } 1 \leq i \leq s \text{ and } x, y\in \calD.
	\end{equation}
\end{proposition}



As a direct corollary of \autoref{prop:Sier-HausDim}, we have the following condition on when the measures of maximal entropy and full dimension coincide.

\begin{corollary}\label{prop:FullShift-MaxFull}
	Let $ K = R(\calD^{\bbN}) $ for some $ \emptyset \neq  \calD \subset \calA $. Then the measures of maximal entropy and full dimension coincide if and only if
	\begin{equation}\label{eq:Match-KP}
		\prod_{i = 1}^{s} Z^{(i)}(\tau_{i}(x))^{\alpha_{i}-1} = \prod_{i = 1}^{s} Z^{(i)}(\tau_{i}(y))^{\alpha_{i}-1} \mFor x, y \in \calD.
	\end{equation}
	When $ s \leq 3 $, \eqref{eq:Match-KP} is equivalent to
	\begin{equation}\label{eq:MaxFull-FullShift}
		\ff{x} = \ff{y} \mFor x, y \in \calD.
	\end{equation}
\end{corollary}

\begin{proof}
	Since the measure of maximal entropy on $ \calD^{\bbN}$ is the Bernoulli measure with marginal $(1/\#\calD, \ldots, 1/\#\calD)$, the first equivalence follows directly from \autoref{prop:Sier-HausDim}. Next we show the second equivalence.
	Since $ f_{1}(x) = Z^{(1)}(x) = 1 $ and $ f_{2}(x) = Z^{(2)}(\tau_{2}(x)) $ for $ x\in \calD$, \eqref{eq:MaxFull-FullShift} and \eqref{eq:Match-KP} are equivalent when $ s \leq 2 $. Next we suppose $ s = 3 $. Let $ x \in \calD$. For $ z  \in \pi_{3}^{-1}(\tau_{3}(x)) $, there exists $ y \in \calD $ such that $ \tau_{2}(y) = z $, and so $ \tau_{3}(y) = \pi_{3}(\tau_{2}(y)) = \pi_{3}(z) = \tau_{3}(x)$. Then either \eqref{eq:MaxFull-FullShift} or \eqref{eq:Match-KP} implies $ f_{2}(x) = f_{2}(y) $, thus $ Z^{(2)}(z) = Z^{(2)}(\tau_{2}(y)) = f_{2}(y) = f_{2}(x)$. Hence by \eqref{eq:def-recur-Zi},
	\begin{equation*}
		Z^{(3)}(\tau_{3}(x)) = \sum_{z \in \pi_{3}^{-1}(\tau_{3}(x))} Z^{(2)}(z)^{\alpha_{2}} = f_{3}(x) f_{2}(x)^{\alpha_{2}}.
	\end{equation*}
	Then
	\begin{equation*}
		\prod_{i = 1}^{3} Z^{(i)}(\tau_{i}(x))^{\alpha_{i}-1} = f_{2}(x)^{\alpha_{2}-1} \left (f_{3}(x) f_{2}(x)^{\alpha_{2}} \right )^{\alpha_{3}-1} = f_{2}(x)^{\theta_{1}-1}f_{3}(x)^{\theta_{2}-1},
	\end{equation*}
	which finishes the proof.
\end{proof}

It is easy to see that \eqref{eq:UniformFiber} is equivalent to \eqref{eq:MaxFull-FullShift} if $ s \leq 2 $, but stronger than \eqref{eq:MaxFull-FullShift} if $ s = 3 $. This motivates us to give the examples in \autoref{subsec:examples}.



\subsection{Proof of \autoref{thm:SierSponge}} \label{subsec:Pf-BM} Let $ X $ be a subshift and $\mu$ be a measure on $X$. For a gauge function $\varphi$, define
\begin{equation*}
	\Theta_{k}^{\varphi}(\mu, x) = \frac{R\mu(Q_{k}(x))}{\varphi(n_{s}^{-k})} \mFor x \in X \text{ and } k\in \bbN.
\end{equation*}
Write $ \Theta_{k}^{\varphi} $ as $ \Theta_{k}^{\gamma} $ when $\varphi(r) = r^{\gamma}$, $ \gamma \geq 0 $. Let $ E[X] $ and $ \Var(X)$ respectively denote the expectation and variance of a random variable $ X $.

\begin{proof}[Proof of \autoref{thm:SierSponge}]
	By \autoref{thm:DimCoin}\ref{itm:A=>B}, it remains to show that if $ 0  < \Haus{\varphi}(K) < \infty $ for some gauge function $ \varphi $, then $ \dimH K = \dimB K $. 
	
	Suppose $ 0  < \Haus{\varphi}(K) < \infty $. Write $ K = R(\calD^{\bbN})$ for some $ \calD \subset \calA $. Let $ \mu $ be the measure of maximal entropy on $ \calD^{\bbN} $, that is, $ \mu$ is the Bernoulli measure with marginal $ ( 1/ \# \calD, \ldots, 1/\# \calD)$. 
	Let $ x = (x_{j})_{j=1}^{\infty} \in \calD^{\bbN}$ be a random sequence with law $\mu$. For $ k \in \bbN $ and $ 1 \leq j \leq k $, define the random variables
	\begin{equation}\label{eq:def-Xj}
		X_{j}^{(k)} = \log \# \tau_{i}^{-1}(\tau_{i}(x_{j})) \mFor 1 \leq i \leq s \text{ and } \tk[i-1] + 1 \leq j \leq \tk[i].
	\end{equation}
	Note that $ x_{1}, \ldots, x_{k}$ are i.i.d.\ random variables with common law $ \sum_{y\in\calD} 1/(\#\calD) \delta_{y} $. Then $ X_{1}^{(k)}, \ldots, X_{k}^{(k)} $ are independent and uniformly bounded.  Write $ E_{k} = \sum_{j=1}^{k} E[ X_{j}^{(k)}]$ and $ V_{k}^{2} = \sum_{j=1}^{k} \Var( X_{j}^{(k)}) $. For $ k \in \bbN $, define
	\begin{equation}\label{eq:def-L-k}
		L(k) = \log \left( (\#\calD)^{k} \varphi(n_{s}^{-k}) \right).
	\end{equation}

	For $ x = (x_{j})_{j=1}^{\infty} \in \calD^{\bbN} $ and $ k \in \bbN $, \eqref{eq:WordsInApproxCube} implies
	\begin{equation}\label{eq:Count-Gamma-k}
		\begin{aligned}
			\# \Gamma_{k}(x) & = \# \bigg\{ I \in \calL_{k}(\calD^{\bbN}) \colon \tau_{i}(I|_{\tk[i-1]}^{\tk[i]} ) = \tau_{i}(x|_{\tk[i-1]}^{\tk[i]}) \text{ for } 1 \leq i \leq s \bigg \} \\
			& =  \prod_{i=1}^{s}\prod_{j=\tk[i-1]+1}^{\tk[i]} \# \tau_{i}^{-1}(\tau_{i}(x_{j})).
		\end{aligned}
	\end{equation}
	Since $ \mu(I) = 1/(\#\calD)^{k}$ for $ I \in \calD^{k}$, \autoref{lem:MassAppCube} shows
	\begin{equation}\label{eq:mass-RmuQk}
		\begin{aligned}
			R\mu(Q_{k}(x)) & = \sum_{I \in \Gamma_{k}(x)} \mu(I) = \frac{ \#\Gamma_{k}(x)}{(\# \calD)^{k}}.
		\end{aligned}
	\end{equation}
	Combining \eqref{eq:def-Xj}, \eqref{eq:def-L-k}, \eqref{eq:Count-Gamma-k} and \eqref{eq:mass-RmuQk} gives
	\begin{equation}\label{eq:DensityRatio-RandVar}
		\Theta_{k}^{\varphi}(\mu, x) = \exp  \left( \sum_{j=1}^{k} X_{j}^{(k)} - L(k) \right) \mFor x \in \calD^{\bbN}.
	\end{equation}
	By \autoref{prop:HausEqParry} and the density theorem for the Hausdorff measures (see e.g.\ \cite[Theorem 6.2]{Mattila1995}),
	\begin{equation}
	\mu \left \{ x \colon \limsup_{k\to\infty} \Theta_{k}^{\varphi}(\mu, x) \approx 1  \right\} = 1
	\end{equation}
	Then by \eqref{eq:DensityRatio-RandVar}, there exists $ M > 0 $ such that
	\begin{equation}\label{eq:ProbOne-Bound}
		\bbP \left \{ \limsup_{k\to\infty} \Abs{\sum_{j=1}^{k} X_{j}^{(k)} - L(k) } < M \right \} = 1
	\end{equation}
	where $ \bbP $ denotes the underlying probability (that is, a probabilistic notation for $ \mu $).
	
	Suppose on the contrary that $ \dimH K < \dimB K $. It follows from \autoref{prop:DimCoin-Sier} that $ \# \tau_{i_{0}}^{-1}(\tau_{i_{0}}(x)) \neq \# \tau_{i_{0}}^{-1}(\tau_{i_{0}}(y))$ for some $ 1 \leq i_{0} \leq s $ and $ x, y \in\calD$. Let $ Y $ be the random variable with law $ \sum_{x\in\calD} \frac{1}{\#\calD} \delta_{\# \tau_{i_{0}}^{-1}(\tau_{i_{0}}(x))}$, then $ \Var(Y) > 0 $. By \eqref{eq:def-Xj}, $Y $ has the same law as $ X_{j}^{(k)}$ for $ \tk[i-1]+1 \leq j \leq \tk[i]$. Then by the independence of $ X_{j}^{(k)}$,
	\begin{equation*}
		V_{k}^{2} \geq \sum_{j=\tk[i-1]+1}^{\tk[i]} \Var(X_{j}^{(k)}) \gtrsim \Var(Y) (\theta_{i} - \theta_{i-1}) k \, \to \infty,
	\end{equation*}
	as $ k \to\infty$. By Lyapunov's Central Limit Theorem (see e.g.\ \cite[Theorem 7.1.2]{Chung2001}), the normalized random sum $ Y_{k} = (\sum_{j=1}^{k} X_{j}^{(k)} - E_{k})/ V_{k}$ converges to the standard normal random variable $ N(0,1) $ in distribution. Since the probability density function of $ N(0,1)$ is uniformly continuous and $ \lim_{k\to\infty} M/V_{k} = 0 $,
	\begin{equation}\label{eq:LimThm-App}
		\begin{aligned}
			\bbP \left\{ \Abs{\sum_{j=1}^{k} X_{j}^{(k)} - L(k) } < M \right\} = \bbP \left\{ \Abs{Y_{k} - \frac{L(k)-E_{k}}{V_{k}} } < \frac{M}{V_{k}} \right\} \to 0
		\end{aligned}
	\end{equation}
	as $ k \to \infty $. This implies
	\begin{equation*}
		\bbP \left \{ \limsup_{k\to\infty} \Abs{\sum_{j=1}^{k} X_{j}^{(k)} - L(k) } < M \right \} \leq \limsup_{k\to\infty} \bbP \left \{  \Abs{\sum_{j=1}^{k} X_{j}^{(k)} - L(k) } < M \right \} = 0,
	\end{equation*}
	which contradicts \eqref{eq:ProbOne-Bound}.
\end{proof}

\begin{remark}\label{rmk:Gen-LimThm}
	The proof of \autoref{thm:SierSponge} works for some other subshift $X$ satisfying weak specification if we can write $\log R\mu(Q_{k}(x))$ as a sequence of random variables such that some limit theorem can be applied to conclude an analog of \eqref{eq:LimThm-App}. Next we give an example to illustrate it. Let $m_{1}, m_{2}, m_{3}$ be integers such that $ m_{1} > m_{2} > m_{3} \geq 2,\, m_{1}, m_{2} \geq 3$. Let $ \calA, \theta_{i}, \pi_{i}, \tau_{i}, 1 \leq i \leq s = 3 $ be defined as in \autoref{subsec:subshifts}. Let
	\begin{equation*}
		\calD = \left \{ \pmat{0 \\ 0 \\ 0}, \pmat{0 \\ 1 \\ 0}, \pmat{1 \\ 1 \\ 0}, \pmat{0 \\ 2 \\ 1}, \pmat{1 \\ 2 \\ 1}, \pmat{2 \\ 2 \\ 1} \right \}.
	\end{equation*}
	Enumerate $\calD$ as $ \{ z_{1}, \ldots, z_{6}\}$ according to the above order.
	Let $ A = (A_{i,j}) $ be a $ 6 \times 6 $ matrix such that  $ A_{i,i} = 0 $ and $ A_{i,j} = 1$ if $ i \neq j$ for $ 1 \leq i, j \leq 6$. Consider the following subshift of finite type
	\begin{equation*}
		\Sigma_{A} = \left\{ (x_{k})_{k=1}^{\infty} \in \calD^{\bbN} \colon \forall\, k \in \bbN, \, \exists\, 1\leq i, j\leq 6, \, x_{k} = z_{i}, x_{k+1} = z_{j}, \, A_{i,j} = 1 \right\}.
	\end{equation*}
	Let $ K = R(\Sigma_{A})$.
	By \cite{Parry1964} (see also \cite{Walters1982}), the unique measure of maximal entropy $\mu$ on $ \Sigma_{A} $ is the Markov measure generated by the probability vector $ (1/6, \ldots, 1/6)$ and the random matrix $ A / 5 $. By computation, $ \tau_{2}\mu $ is the Markov measure on $ \tau_{2}(\Sigma_{A})$ generated by $ (1/6, 1/3, 1/2)$ and $ \begin{pmatrix}
		0 & 2/5 & 3/5 \\
		1/5& 1/5 & 3/5 \\
		1/5& 2/5 & 2/5
	\end{pmatrix} $. Since the computation shows that $ \tau_{2}\mu$ is not the measure of maximal entropy on $ \tau_{2}(\Sigma_{A})$, it follows from  \autoref{thm:Ext-KP-DimCoin}\ref{itm:ParryChain} that $ \dimH K < \dimB K $. By the definition of $ \calD $, a similar estimate as in the proof of \autoref{thm:SierSponge} gives
	\begin{equation*}
		R\mu(Q_{k}(x)) \approx \theta_{1} k + 3 (1-\theta_{2}) k + \sum_{j = \tk[1] + 1}^{\tk[2]} Y((\tau_{2}x)_{j})
	\end{equation*}
	where $ Y \colon \tau_{2}\calD \to \{1,2,3\} $ is a function defined as $ Y((0, 0)^{t}) = 1,  Y((1, 0)^{t}) = 2$ and $ Y((2, 1)^{t}) = 3 $. Since the law of $\tau_{2}x$ is the Markov measure $ \tau_{2}\mu $ if $ x $ is a random sequence with law $ \mu $, we can apply the central limit theorem for Markov chains (see e.g.\ \cite[Example 8.3.4]{Durrett2019}) and adapt the proof of \autoref{thm:SierSponge} to conclude that $ \haus{\varphi}{K} = 0 $ or $\infty$ for every gauge function $ \varphi $.
\end{remark}

\subsection{Infinite Hausdorff measure} In this subsection, we provide a sufficient condition (see \autoref{prop:CriterionInfHaus}) for \BMsponge[s] to have infinite Hausdorff measures at their respective Hausdorff dimensions. This is based on the method of Peres~\cite{Peres1994} and a relation (see \autoref{lem:Sum-delta-i}) satisfied by the measure of full dimension. In \cite[Section 5]{Peres1994}, Peres mentioned that \cite[Theorem 1 and Proposition 2]{Peres1994} extend to \BMsponge[s]. However, he did not give a detailed justification. In fact, the proof of \cite[Proposition 2]{Peres1994} does not work for \BMsponge[s] according to the examples in \autoref{subsec:examples}. For the extension of \cite[Theorem 1]{Peres1994}, it is not clear how to construct a measure with zero density like that in the proof of \cite[Theorem 1]{Peres1994} since the measure of full dimension becomes more complicated when $ s \geq 3 $ (see \autoref{prop:Sier-HausDim}).

Throughout this subsection we fix $ \emptyset \neq \calD \subset \calA $ and let $ K = R(\calD^{\bbN})$. Let $p = (p(x))_{x\in\calD} $ be the probability vector given in \eqref{eq:McMullen-vec}. Write $ \gamma = \dimH K $. For a finite set $ \calB $ and probability vectors $ q = (q(x))_{x \in \calB}$ and $ q' = (q'(x))_{x\in \calB}$, following \cite{Peres1994} we define
\begin{equation}\label{eq:Def-Delta}
	\Delta(q\| q') = \sum_{x \in \calB} (q(x)- q'(x)) \log q(x),
\end{equation}
where we require $ q'(x) = 0 $ if $ q(x) = 0 $. 

The following proposition is essentially contained in \cite{Peres1994}. For completeness, we include a proof of \autoref{prop:PeresInRd} in \autoref{sec:PeresRd}.

\begin{proposition}[{Peres~\cite{Peres1994}}]\label{prop:PeresInRd} If there exist a probability vector $ q $ on $\calD $ and $ c > 0 $ such that
	\begin{equation}\label{eq:Cond-Peres}
		\sum_{i=1}^{s} \Delta(\tau_{i}p\| \tau_{i} q) \int_{\theta_{i-1}k}^{\theta_{i}k} \frac{1}{\log t} \,d t \geq c \frac{k}{(\log k)^{2}} \quad\text{ for large } k \in \bbN,
	\end{equation}
	then $ \haus{\varphi}{K} = \infty $, where
	\begin{equation}\label{eq:peres-gauge}
		\varphi(r) = r^{\gamma} \exp\left(\wt{c} \frac{\abs{\log r}}{(\log \abs{\log r})^{2}} \right), \quad r > 0,
	\end{equation}
	for some $ \wt{c} > 0 $ depending on $ c $ explicitly. In particular, $ \haus{\gamma}{K} = \infty$.
\end{proposition}

In the remaining part of this subsection, we prove the following proposition which is a sufficient condition such that the Hausdorff measures of Bedford\nobreakdash-McMullen sponges are infinite at their Hausdorff dimensions. For  a probability vector $ q = (q(x))_{x\in\calB}$ on a finite set $\calB $, we say $q $ is \textit{uniform} if $q = (1/\#\calB, \ldots, 1/\#\calB)$.

\begin{proposition}\label{prop:CriterionInfHaus}
	If 
	\begin{equation}\label{eq:Cond-InfHaus}
		 1 \leq \# \{ 1 \leq i \leq s \colon \tau_{i} p \text{ is not uniform}\} \leq 2,
	\end{equation}
	then $ \haus{\dimH K}{K} = \infty $.
\end{proposition}

\autoref{prop:CriterionInfHaus} extends some results in \cite{Peres1994} since for \BMcarpet[s], the upper bound in \eqref{eq:Cond-InfHaus} trivially hold when $ d = 2 $ and the lower bound in \eqref{eq:Cond-InfHaus} follows from $ \dimH K < \dimB K $ and \autoref{thm:Ext-KP-DimCoin}\ref{itm:ParryChain}. Next we give two lemmas for the proof of \autoref{prop:CriterionInfHaus}.

\begin{lemma}\label{lem:Delta-properties}
	If $ q $ is uniform, then $ \Delta(q \| q') = 0 $; if $ q' $ is uniform but $ q $ is not uniform, then $\Delta(q\| q') > 0 $.
\end{lemma}
\begin{proof}
	The first statement is by a direct computation. The second statement follows from
	\begin{equation*}
		\Delta(q\| q') = \frac{1}{2\#\calB} \sum_{x\in\calB} \sum_{y\in\calB} (q(x)-q'(y))(\log q(x)- \log q'(y)) > 0
	\end{equation*}
	when $ q' $ is uniform.
\end{proof}

The following lemma is the key ingredient for the verification of \eqref{eq:Cond-Peres} when $ s \geq 3 $.

\begin{lemma}\label{lem:Sum-delta-i}
 For every probability vector $ q $ on $ \calD $,
	\begin{equation}\label{eq:sum-Delta-i}
		\sum_{i=1}^{s} (\theta_{i}- \theta_{i-1})\Delta( \tau_{i}p\| \tau_{i} q) = 0.
	\end{equation}
\end{lemma}
\begin{proof}
	For $ 1\leq i \leq s $, write $ p_{i} = \tau_{i}p $, $ q_{i} = \tau_{i} q $, and $\Delta_{i} = \Delta(p_{i}\| q_{i})$. By \eqref{eq:Def-Delta},
	\begin{equation*}
		\begin{aligned}
			\Delta_{i} & = \sum_{x \in \calD_{i}} (p_{i}(x)-q_{i}(x)) \log p_{i}(x) \\
			& = \sum_{x\in\calD_{i}} \log p_{i}(x) \sum_{y\in\tau_{i}^{-1}(x)} (p_{1}(y) - q_{1}(y)) & \text{by }p_{i}=\tau_{i}p_{1}, q_{i}=\tau_{i}q_{1}\\
			& = \sum_{x\in\calD_{i}}  \sum_{y\in\tau_{i}^{-1}(x)} (p_{1}(y) - q_{1}(y))\log p_{i}(\tau_{i}(y)) & \text{by } \tau_{i}(y) = x \text{ for } y \in \tau_{i}^{-1}(x)\\ 
			& = \sum_{y\in\calD_{1}} (p_{1}(y) - q_{1}(y))\log p_{i}(\tau_{i}(y)).
		\end{aligned}
	\end{equation*}
	Then
	\begin{equation*}
		\begin{aligned}
			\sum_{i=1}^{s} (\theta_{i}- \theta_{i-1})\Delta_{i} & = \sum_{x\in\calD_{1}} (p_{1}(x) - q_{1}(x)) \sum_{i=1}^{s} (\theta_{i}-\theta_{i-1}) \log p_{i}(\tau_{i}(x)).
		\end{aligned}
	\end{equation*}
	Since $ \sum_{x\in\calD_{1}}p_{1}(x) = \sum_{x\in\calD_{1}} q_{1}(x) =  1$, it suffices to show 
	\begin{equation}
		\sum_{i=1}^{s} (\theta_{i}-\theta_{i-1}) \log p_{i}(\tau_{i}(x)) = - \log Z \mFor x \in \calD_{1}.
	\end{equation}
	By \eqref{eq:McMullen-vec}, for $ x \in \calD_{1}$ and $ 2 \leq i \leq s$,
	\begin{equation*}
		p_{1}(x) = p_{i}(\tau_{i}(x)) \frac{1}{Z^{(i)}(\tau_{i}(x))} \prod_{j=1}^{i-1} Z^{(i)}(\tau_{j}(x))^{\alpha_{j}-1}.
	\end{equation*}
	Thus
	\begin{equation*}
		\begin{aligned}
			\sum_{i=1}^{s} &  (\theta_{i}-\theta_{i-1}) \log p_{i}(\tau_{i}(x)) \\
			& = \log p_{1}(x) - \sum_{i=2}^{s} (\theta_{i-1}-\theta_{i}) Z^{(i)}(\tau_{i}(x)) - \sum_{i=2}^{s}\sum_{j=1}^{i-1} (\theta_{i}-\theta_{i-1}) (\alpha_{j}-1)\log Z^{(j)}(\tau_{j}(x)) \\
			& = \log p_{1}(x) - \sum_{i=2}^{s} (\theta_{i-1}-\theta_{i}) Z^{(i)}(\tau_{i}(x)) - \sum_{j=1}^{s-1}\sum_{i=j+1}^{s} (\theta_{i}-\theta_{i-1}) (\alpha_{j}-1)\log Z^{(j)}(\tau_{j}(x)) \\
			& = \log p_{1}(x) - \sum_{i=2}^{s} (\theta_{i-1}-\theta_{i}) Z^{(i)}(\tau_{i}(x)) - \sum_{j=1}^{s-1}(1-\theta_{j}) (\alpha_{j}-1)\log Z^{(j)}(\tau_{j}(x)) \\
			& = \log p_{1}(x) - \sum_{i=1}^{s} (\alpha_{i}-1) Z^{(i)}(\tau_{i}(x)) \\
			& = - \log Z,
		\end{aligned} 
	\end{equation*}
	where the second equality is by changing the index of $i, j$; the second last equality is by $ Z^{(1)} \equiv 1 $ and $ \theta_{i}\alpha_{i} = \theta_{i-1}$ for $ 1 \leq i \leq s$; the last equality is by \eqref{eq:McMullen-vec}.
\end{proof}

Now we are ready to prove \autoref{prop:CriterionInfHaus}.

\begin{proof}[Proof of \autoref{prop:CriterionInfHaus}]
	For a probability vector $ q $ on $ \calD $ and $ 1 \leq i \leq s $, write $ p_{i} = \tau_{i}p $, $ q_{i} = \tau_{i} q $, and $\Delta_{i} = \Delta(p_{i}\| q_{i})$. Define $ I = \{ 1 \leq i \leq s \colon p_{i} \text{ is not uniform} \}$. It follows from \autoref{lem:Delta-properties} that $ \Delta_{i} = 0 $ for $ i \notin I $. By \autoref{lem:Sum-delta-i},
	\begin{equation}\label{eq:Reduce-SumDelta-i}
		\sum_{i \in I} (\theta_{i}-\theta_{i-1}) \Delta_{i} = \sum_{i=1}^{s} (\theta_{i}-\theta_{i-1}) \Delta_{i} = 0.
	\end{equation}
	
	We first show $ \# I \neq 1 $. Suppose otherwise that $ I = \{i_{1}\}$. Take a probability vector $ q $ on $\calD$ such that $ \tau_{i_{1}}q $ is uniform. Then $ \Delta_{i_{1}} > 0 $ by \autoref{lem:Delta-properties}, which contradicts \eqref{eq:Reduce-SumDelta-i}. Then $ \# I = 2 $ by \eqref{eq:Cond-InfHaus}.
	
	Write $ I = \{ i_{1}, i_{2}\}$ for some $ 1 \leq i_{1} < i_{2} \leq s $. Take a probability vector $ q $ on $\calD$ such that $ \tau_{i_{1}}q $ is uniform. Then $ \Delta_{i_{1}} > 0 $ by \autoref{lem:Delta-properties}. By \eqref{eq:Reduce-SumDelta-i}
	\begin{equation*}
	 \Delta_{i_{2}} = - \frac{\theta_{i_{1}}-\theta_{i_{1}-1}}{\theta_{i_{2}}-\theta_{i_{2}-1}}\Delta_{i_{1}} < 0.
	\end{equation*}
	Since $ \Delta_{i} = 0 $ for $i \neq I$,
	\begin{equation}\label{eq:Reduce-Cond-Peres}
		\begin{aligned}
			\sum_{i=1}^{s} & \Delta_{i} \int_{\theta_{i-1}k}^{\theta_{i}k} \frac{1}{\log t} \,d t \\
			& = \Delta_{i_{1}}(\theta_{i_{1}}-\theta_{i_{1}-1})\left( \frac{1}{\theta_{i_{1}}-\theta_{i_{1}-1}} \int_{\theta_{i_{1}-1}k}^{\theta_{i_{1}}k}\frac{1}{\log t} \, dt  -  \frac{1}{\theta_{i_{2}}-\theta_{i_{2}-1}} \int_{\theta_{i_{2}-1}k}^{\theta_{i_{2}}k}\frac{1}{\log t} \, dt \right).
		\end{aligned}
	\end{equation}
	By a change of variable, for $ 0 < a < b \leq c < d$ there exists $ \delta > 0 $ such that
	\begin{equation}\label{eq:IntegralDiff}
		\frac{1}{b-a} \int_{ak}^{bk} \frac{1}{\log t} \, dt - \frac{1}{d-c} \int_{ck}^{dk} \frac{1}{\log t} \, dt \geq \delta \frac{k}{(\log k)^{2}} \mFor k \in \bbN \intxn [2/a,\infty).
	\end{equation}
	Combining \eqref{eq:Reduce-Cond-Peres} and \eqref{eq:IntegralDiff} verifies \eqref{eq:Cond-Peres}. Finally, the proof is completed by \autoref{prop:PeresInRd}.
\end{proof}

\begin{remark}
	From the proof of \autoref{prop:CriterionInfHaus} we see that
	\begin{equation}\label{eq:nonuniform!=1}
		\# \{ 1 \leq i \leq s \colon \tau_{i}p \text{ is not uniform}\} \neq 1.
	\end{equation}
	Based on \eqref{eq:nonuniform!=1} there is an alternative way to prove that for a Bedford-McMullen carpet $K$, $ 0 < \haus{\varphi}{K} < \infty $ implies $ \dimH K = \dimB K $. Suppose otherwise $ \dimH K < \dimB K $. Combining \autoref{prop:HausEqParry} and \autoref{thm:Ext-KP-DimCoin}\ref{itm:ParryChain} gives $ \# \{ 1 \leq i \leq s \colon \tau_{i}p \text{ is not uniform}\} = 1 $ which contradicts \eqref{eq:nonuniform!=1}.
\end{remark}

\section{Examples}\label{subsec:examples}

In this section, we present several \BMsponge[s]\ $K=R(\calD^{\bbN})$, $ \emptyset \neq \calD \subset \calA $ such that $ \dimH R\mu = \dimH K < \dimB K $, where $ \mu $ is the measure of maximal entropy on $\calD^{\bbN}$. For a finite set $ \calB$ and a probability vector $ p = (p(x))_{x\in\calB} $, let $ \eta_{p} $ denote the Bernoulli measure on $ \calB^{\bbN}$ with marginal $p$.
 
\begin{example}\label{ex:d=3}
	Let
	\begin{equation*}
		\Lambda  = \diag(64, 16, 8)
	\end{equation*}
	and
	\begin{equation*}
		\calD = \left \{ \pmat{0 \\ 0 \\ 0}, \pmat{0 \\ 1 \\ 0}, \pmat{0 \\ 2 \\ 0}, \pmat{0 \\ 3 \\ 0}, \pmat{0 \\ 0 \\ 1}, \pmat{1 \\ 0 \\ 1} \right \}.
	\end{equation*}
	Define $ K = R( \calD^{\bbN} ) $. By \autoref{prop:BoxDim} and \autoref{prop:Sier-HausDim},
	\begin{equation*}
		\dimB K  = \frac{\log 360}{12 \log 2} \mAnd \dimH K  = \frac{\log 18}{6 \log 2}.
	\end{equation*}
	The measure of maximal entropy is $ \eta_{p} $ with $ p = p_{1} = (1/6, 1/6, 1/6, 1/6, 1/6, 1/6 ) $. 
	Note that $ \tau_{2}\eta_{p_{1}} = \eta_{p_{2}} $ and $ \tau_{3}\eta_{p_{1}} = \eta_{p_{3}}$, where $ p_{2} = (1/6, 1/6, 1/6, 1/6, 1/3)$ and $ p_{3} = (2/3, 1/3) $. By \autoref{thm:dimLY-ExistFullDim},
	\begin{equation*}
		\begin{aligned}
			\dimH R \eta_{p} = \dLY{\eta_{p}} = \frac{\log 18}{6 \log 2}.
		\end{aligned} 
 	\end{equation*}
	Hence
	\begin{equation*}
		\dimH R\eta_{p} = \dimH K < \dimB K.
	\end{equation*}
	Since $\#\{ 1 \leq i \leq 3 \colon p_{i} \text{ is not uniform}\} = 2$, it follows from \autoref{prop:CriterionInfHaus} that
	\begin{equation*}
		 \haus{\dimH K}{K} = \infty.
	\end{equation*}
\end{example}

Next we give similar examples when $ d \geq 4 $.

\begin{example}\label{ex:d>=4}
	For $ d \geq 4 $, let
	\begin{equation*}
		\Lambda = \diag ( 2^{2^{d-1}}, 2^{2^{d-2}}, \ldots, 4, 2)
	\end{equation*}
	and
	\begin{equation*}
		\calD = \left \{ \pmat{\{0, 1, 2, 3\}\\ \{0\} \\ \{0,1\} \\ \{0\} \\ \{ 0, 1\} \\ \vdots \\ \{0, 1\}}, \pmat{\{0\} \\ \{0, 1, 2, 3, 4, 5, 6, 7\} \\ \{0\} \\ \{1\} \\ \{ 0, 1\} \\ \vdots \\ \{ 0, 1\} } \right \}.
	\end{equation*}
	Write the  left and right collections of digits above respectively as $ \calD_{1} $ and $\calD_{2}$. Then $ \calD = \calD_{1} \sqcup \calD_{2} $. Since for $ x \in \calD_{1}$ and $y \in \calD_{2}$,
	\begin{equation*}
		\# \pi_{2}^{-1}(\tau_{2}(x)) = 4 \neq 1 = \# \pi_{2}^{-1}(\tau_{2}(y)),
	\end{equation*}
	\autoref{prop:DimCoin-Sier} implies $ \dimH K < \dimB K $. Note that $ \alpha_{i} = 1/2 $ for $ 2 \leq i \leq d $. Then for $ x \in \calD_{1} $ and $ y \in \calD_{2} $, by \eqref{eq:def-recur-Zi},
	\begin{equation}\label{eq:Ex2-i<=4}
		\begin{aligned}
		Z^{(2)}(\tau_{2}(x)) = 4 & \mAnd Z^{(2)}(\tau_{2}(y)) = 1 \\
		Z^{(3)}(\tau_{3}(x)) = 2 & \mAnd Z^{(3)}(\tau_{3}(y)) = 8 \\
		Z^{(4)}(\tau_{4}(x)) = 2\sqrt{2} & \mAnd Z^{(4)}(\tau_{4}(y)) = 2\sqrt{2}.
	\end{aligned}
	\end{equation}
	If $ i > 4 $ and $ x , y \in \calD $, since $ \# \pi_{i}^{-1}(\tau_{i}(x)) = \# \pi_{i}^{-1}(\tau_{i}(y)) $,
	\begin{equation}\label{eq:Ex2-i>4}
		Z^{(i)}(\tau_{i}(x)) = Z^{(i)}(\tau_{i}(y)).
	\end{equation}
	By computation,
	\begin{equation*}
		\prod_{i=1}^{3} Z^{(i)}(\tau_{i}(x))^{\alpha_{i}-1} = 2^{-3/2} \mFor x \in \calD.
	\end{equation*}
	Then \autoref{prop:FullShift-MaxFull} shows that $ \dimH R\eta_{p} = \dimH K $, where $ p = (1/\#\calD, \ldots, 1/\#\calD)$. Hence $\dimH R\eta_{p} = \dimH K < \dimB K$.  Since \eqref{eq:McMullen-vec} implies that $ p_{1} $ and $ \tau_{i}p, i \geq 4 $ are uniform while $ \tau_{2}p $ and $ \tau_{3} p$ are not uniform, it follows from \autoref{prop:CriterionInfHaus} that $ \haus{\dimH K}{K} = \infty $.
\end{example}

\appendix

\section{Proof of \autoref{prop:PeresInRd}}\label{sec:PeresRd}
\begin{proof}[Proof of \autoref{prop:PeresInRd}] The proof is essentially contained in \cite{Peres1994}. For completeness we give a sketch of the proof. The strategy is to construct a measure $ \nu $ on $ \calD^{\bbN} $ such that $ \limsup_{k\to\infty} R\nu(Q_{k}(x))/\varphi(n_{s}^{-k}) = 0 $ for $\nu$-a.e.\ $ x $.
	
	For $ 1\leq i \leq s $, write $ p_{i} = \tau_{i}p $, $ q_{i} = \tau_{i} q $, and $\Delta_{i} = \Delta(p_{i}\| q_{i})$.
	For $ \delta > 0 $ and $ j \geq 2 $, following \cite{Peres1994} we define
	\begin{equation*}
		\xi^{(j)} = \left (1-\frac{\delta}{\log j}\right) p + \frac{\delta}{\log j} q.
	\end{equation*}
	and $ \xi^{(1)} = p $. Define the product measure on $\calD^{\bbN}$ by
	\begin{equation*}
		\nu = \prod_{j = 1}^{\infty} \xi^{(j)}.
	\end{equation*}
	For $ 1 \leq i \leq s $ and $ 1 \leq j \leq k$, $k\in \bbN$, let $ X_{i,j}^{(k)} $ be the random variable on $ \calD_{j} $ with law $ \sum_{y \in \calD_{i}} \tau_{i}\xi^{(j)}(y) \, \delta_{\log \tau_{i}\xi^{(j)}(y)}$. Then $ X_{i,j}^{(k)}$ is well defined since for $ y \in \calD_{i}$ there is $ x \in \calD$ with $ \tau_{i}(x) = y $ such that $ \tau_{i}\xi^{(j)}(y) \geq (1-\frac{\delta}{\log j}) p(x) > 0 $ by \eqref{eq:McMullen-vec}. By \eqref{eq:WordsInApproxCube} and \autoref{lem:MassAppCube}, for $ k \in \bbN $ and  $ x = (x_{j}) \in \calD^{\bbN}$,
	\begin{equation}\label{eq:RandSumLogMeas}
		\log R\nu(Q_{k}(x)) = \sum_{i=1}^{s} \sum_{j = \tk[i-1] + 1}^{\tk[i]} \log \tau_{i}\xi^{(j)}(\tau_{i}(x_{j})) = \sum_{i=1}^{s} \sum_{j = \tk[i-1] + 1}^{\tk[i]} X_{i,j}^{(k)}.
	\end{equation}
	Then $ \log \nu(Q_{k}(x))$ is a sum of independent variables. By \eqref{eq:Cond-Peres} and \autoref{lem:Delta-properties}, there is $ 1 \leq i_{0} \leq s $ such that $ \tau_{i_{0}}p = p_{i_{0}} $ is not uniform. Let $Y_{i_{0}}$ be the random variable on $ \calD_{i_{0}}$ with the law $\sum_{y \in \calD_{i_{0}}} p_{i_{0}}(y) \, \delta_{\log p_{i_{0}}(y)}$. Then $ \Var(X_{i_{0},j}^{(k)}) \geq \Var(Y_{i_{0}}) / 2 > 0 $ for $ j $ large enough since $ \tau_{i_{0}}\xi^{(j)} = \left (1-\frac{\delta}{\log j}\right) p_{i_{0}} + \frac{\delta}{\log j} q_{i_{0}} $. Then for $k $ large,
	\begin{equation*}
		\Var(\log R\nu(Q_{k}(x))) \geq \sum_{j = (\tk[i_{0}] + \tk[i_{0}-1]) / 2}^{\tk[i_{0}]} \Var(X_{i_{0},j}^{k}) \geq \frac{\Var(Y_{i_{0}})(\theta_{i_{0}} - \theta_{i_{0}-1})}{4} k \gtrsim k.
	\end{equation*}
	Thus the law of iterated logarithm (see, e.g.\ \cite[Theorem 7.5.1]{Chung2001}) implies that
	\begin{equation}\label{eq:App-LLT}
		\log R\nu(Q_{k}(x)) \leq E_{k}(\nu) + O\left(\sqrt{k\log\log k}\right) \mFor \text{$\nu$-a.e.\ } x.
	\end{equation}
	where $E_{k}(\nu) :=  \int \log R\nu(Q_{k}(x)) \, d\nu(x)$. By Taylor's theorem, for probability vectors $ q, q'$ and $\varepsilon > 0 $,
	\begin{equation}\label{eq:Taylor2Shannon}
		H((1-\varepsilon)q + \varepsilon q') = H(q) + \varepsilon \Delta(q\|q') + O(\varepsilon^{2}),
	\end{equation}
	where $H(\cdot)$ is the \textit{Shannon entropy}, that is, for a probability vector $ q = (q(x))_{x\in\calB} $,
	\begin{equation*}
		H(q) = \sum_{x\in \calB} - q(x) \log q(x).
	\end{equation*}
	Then
	\begin{equation*}
		\begin{aligned}
			E_{k}(\nu) & = \sum_{i=1}^{s} \sum_{j = \tk[i-1] + 1}^{\tk[i]} \sum_{x\in \calD_{i}} \tau_{i}\xi^{(j)}(x) \log \tau_{i}\xi^{(j)}(x) \\
			& = \sum_{i=1}^{s} \sum_{j = \tk[i-1] + 1}^{\tk[i]} - H(\tau_{i}\xi^{(j)}) \\
			& = \sum_{i=1}^{s} \sum_{j = \tk[i-1] + 1}^{\tk[i]} - H\left( \left (1-\frac{\delta}{\log j}\right) p_{i} + \frac{\delta}{\log j} q_{i} \right  ) \\
			& = \sum_{i=1}^{s} \sum_{j = \tk[i-1] + 1}^{\tk[i]} - H\left( p_{i} \right  ) - \delta \Delta_{i} \frac{1}{\log j} + O\left( \frac{\delta^{2}}{(\log j)^{2}}\right) \\
			& = - k \log Z  - \delta \sum_{i=1}^{s} \Delta_{i} \sum_{j=\tk[i]+1}^{\tk[i]} \frac{1}{\log j} + O(1) + \delta^{2} O(1) ,  \\
		\end{aligned}
	\end{equation*}
	where the last equality is by \autoref{prop:Sier-HausDim} and $ \sum_{j\geq 2} 1/(\log j)^{2} < \infty $. Since
	\begin{equation*}
		\sum_{j = a}^{b} \frac{1}{\log j}= \int_{a}^{b} \frac{1}{\log t}\, d t + O(1) \mFor 2 \leq a < b,
	\end{equation*}
	we have
	\begin{equation}\label{eq:Ek-Expect}
		E_{k}(\nu) = - k \log Z  - \delta \sum_{i=1}^{s} \Delta_{i} \int_{\theta_{i-1}k}^{\theta_{i}k} \frac{1}{\log t} \,d t + O(1) .
	\end{equation}
	
	Let $ \wt{c} > 0 $ be small such that $ \wt{c} M < c\delta / 2 $, where $ M > 0 $ is large such that for $ k \in \bbN \intxn [2, \infty) $, $
	\abs{\log n_{s}^{-k}} / (\log \abs{\log n_{s}^{-k}})^{2} \leq M k/(\log k)^{2}$. Take $ \varphi $ as in \eqref{eq:peres-gauge}. Combining \eqref{eq:Ek-Expect}, \eqref{eq:Cond-Peres}, \eqref{eq:peres-gauge} and \eqref{eq:App-LLT} gives
	\begin{equation*}
		\limsup_{k\to\infty} \log R\nu(Q_{k}(x)) - \log \varphi(n_{s}^{-k}) = - \infty \mFor \text{$\nu$-a.e.,x}.
	\end{equation*}
	By \autoref{lem:RogersTaylor}, this shows $\haus{\varphi}{K} = \infty $.
\end{proof}

%


%



%
%

\end{document}